\documentclass[11pt,reqno]{amsart}
\usepackage{amsmath,amssymb}
\usepackage{dsfont}
\usepackage{xcolor}
\usepackage{indentfirst,csquotes}
\usepackage{mathtools}
\usepackage{MnSymbol}
\usepackage{verbatim}
\usepackage{upgreek}
\usepackage{enumerate}
\usepackage{amssymb}
\usepackage{pifont}
\usepackage{caption}
\usepackage{subcaption}
\usepackage{graphicx}
\usepackage{bm}
\usepackage[margin=1in]{geometry}
\usepackage[colorlinks=true,citecolor=magenta,linkcolor=blue]{hyperref}

\newcommand{\eps}{\varepsilon}
\newcommand{\eq}{\mathrm{eq}}
\newcommand{\norm}[1]{\left\|#1\right\|}

\newcommand{\R}{\mathbb{R}}

\newcommand{\dd}{\mathrm{d}}
\newcommand{\init}{\mathrm{in}}

\newcommand{\blfootnote}[1]{%
  \begingroup
  \renewcommand\thefootnote{}\footnote{#1}%
  \addtocounter{footnote}{-1}%
  \endgroup
}

\newcommand{\smallO}{
  \mathchoice
    {{\scriptstyle\mathcal{O}}}
    {{\scriptstyle\mathcal{O}}}
    {{\scriptscriptstyle\mathcal{O}}}
    {\scalebox{.7}{$\scriptscriptstyle\mathcal{O}$}}
  }

\newtheorem{theorem}{Theorem}
\newtheorem{lemma}{Lemma}
\newtheorem{proposition}{Proposition}
\newtheorem{conjecture}{Conjecture}
\newtheorem{remark}{Remark}
\newtheorem{corollary}{Corollary}

\makeatletter
\def\blfootnote{\xdef\@thefnmark{}\@footnotetext}
\makeatother

\author[A. Bondesan]{Andrea Bondesan}
\address{Andrea Bondesan \hfill\break
	Department of Mathematical, Physical and Computer Sciences \hfill\break 
    	University of Parma \hfill\break
	Parco Area delle Scienze 53/A, 43124 Parma, Italy \vspace*{2mm}}
 
\email{andrea.bondesan@unipr.it \vspace*{5mm}}

\author[J. Borsotti]{Jacopo Borsotti}
\address{Jacopo Borsotti \hfill\break
	Department of Mathematical, Physical and Computer Sciences \hfill\break
	University of Parma \hfill\break
	Parco Area delle Scienze 53/A, 43124 Parma, Italy \vspace*{2mm}}
\email{jacopo.borsotti@unipr.it \vspace*{5mm}}

\title[Opinion dynamics driven by social activity]{Control strategies and trends to equilibrium for kinetic models of opinion dynamics driven by social activity}

\begin{document}

\begin{abstract}
    We introduce new kinetic equations modeling opinion dynamics inside a population of individuals, whose propensity to interact with each other is described by their level of social activity. We show that opinion polarization can arise among agents with a low activity level, while active ones develop a consensus, highlighting the importance of social interactions to prevent the formation of extreme opinions. Moreover, we present a realistic control strategy aimed at reducing the number of inactive agents and increasing the number of socially active ones. At last, we prove several (weak and strong) convergence to equilibrium results for such controlled model. In particular, by considering additional interactions between individuals and opinion leaders capable of steering the average opinion of the population, we use entropy method-like techniques to estimate the relaxation toward equilibrium of solutions to a Fokker--Planck equation with time-dependent coefficients.
\end{abstract}

\maketitle

\tableofcontents

\section{Introduction} 

    \noindent In recent years, the modeling of opinion dynamics has attracted the interest of researchers from different scientific fields, such as social sciences and mathematics. This growing popularity follows from the pervasive influence of social networks and global events that have shown how collective opinions can shape our society. For instance, the public debate around the use of vaccines during the COVID-19 pandemic has been characterized by the emergence of a significant number of conspiracy theories and misinformation, highlighting the importance of studying the processes behind opinion formation \cite{28}. In particular, it would be valuable to understand how collective opinions can polarize toward potentially dangerous, extreme opinions, or can converge toward a consensus. It is important to remark that individuals are constantly influenced by other people and by the media, and the evolution of their opinions strongly depends on their propensity to interact and discuss with others. As an example, one can think of the current political scenario in various European countries, where the phenomenon of abstentionism has continued growing in the past few years \cite{12, 14, 13}, and an increasingly large part of the population has lost interest in finding out about the latest relevant news.  
    
    Over time, several different models have been developed to mimic opinion dynamics inside a society composed of a large number of individuals. Early models \cite{22,23,24,25,26,27} aimed to capture collective behavior using techniques borrowed from statistical mechanics: the society was represented as a graph, while the agents as nodes interacting with their neighbors. Later, thanks to Toscani’s seminal work \cite{8}, kinetic models emerged as a powerful alternative to the previous approaches, finding various applications for instance in the contexts of opinion formation \cite{MR3945372,BouSal,3,DurWol,MR3945374,6}, wealth distribution in societies \cite{BisSpiTos,11,FurPulTerTos,MR4093946,15}, or epidemics spreading \cite{2,DelLoyTos,DimPerTosZan,16}. Kinetic models describe opinion formation through binary interactions between agents. It is typically assumed that an individual's opinion evolves according to two mechanisms: compromise dynamics, where agents vary their opinions through interactions with others, and self-thinking dynamics. It is also possible to model opinion variations arising from interactions with opinion leaders (e.g., the media or highly influential individuals), which are supposed to be unaffected by individuals' opinions \cite{2}.
    
    The main assumption behind these kinetic models is that when two agents interact, only small variations of their opinions can occur. Considering this quasi-invariant regime of opinions, Toscani \cite{8} showed that their dynamics can be described by Fokker--Planck-type equations. The great advantage is that the equilibria of the latter can be computed analytically, and hence offer a clear description of key phenomena like consensus formation and opinion polarization. These equilibria are crucial to understand how extreme opinions emerge within a society. Preventing their formation is particularly important, especially in contexts where polarized views can have harmful consequences. For example, in \cite{16} it is shown that opinion polarization against protective behaviors during an epidemic can lead to a significant increase in the number of infections, highlighting the real-world impact of opinion dynamics. 

    Inspired by the work of Moussa\"id \cite{1}, in this paper we present a new kinetic model describing opinion formation in a population where each individual's propensity to interact with others is characterized by a (variable) social activity. In particular, each agent is associated with two microscopic traits being an opinion $w$ and a level of activity $A$, and depending on the value of $A$ we distinguish between three classes of individuals: active agents that are socially involved and are likely to interact with others, inactive agents who have no propensity to discuss with others and therefore have a small probability to interact, undecided agents who display a behavior in-between the two previous groups of individuals. We also expect that the activity level $A$ of an individual increases as they interact with others, while it gradually fades away over time when no interactions occur. Moreover, we assume that $A$ evolves independently of $w$, but it still modulates the exchanges of opinions. The main practical result of this paper is the demonstration that opinion polarization is more likely to occur among inactive agents. Indeed, active ones are more prone to reach a consensus, highlighting the importance of social interactions to prevent the formation of dangerous extreme opinions outside a negligible portion of the population. This phenomenon is made precise by analyzing the partial equilibria of the model, which are described by Beta distributions. We therefore introduce a realistic control strategy with the aim to increase the overall number of active individuals and reduce the number of inactive ones. This strategy is designed to reproduce any sudden growth or decline of interest in a certain piece of information, occurring among a fraction of the population. We will show that when those two phenomena are correctly balanced the control strategy is effective, while if such balance is not met an opposite outcome might arise. From a mathematical point of view, our most interesting results concern the study of (weak and strong) convergence to equilibrium for solutions to the controlled model. In particular, considering the effects of additional interactions with opinion leaders, we exploit the relative entropy method \cite{20} to prove the relaxation to equilibrium for a Fokker--Planck equation with time-dependent coefficients, whose adjoint is the well-known Wright--Fisher-type model \cite{CheStr,EpsMaz1,EpsMaz2}. 

    The paper is organized as follows. In Section \ref{sec1} we introduce our model, considering both interactions among the agents and with opinion leaders. In particular, starting from a Boltzmann description of these microscopic interactions, we derive the corresponding Fokker--Planck approximation in a quasi-invariant regime of the parameters, allowing us to study the evolution of the relevant macroscopic quantities. We also investigate how the individuals' activity level is related to opinion polarization and consensus formation phenomena. At last, we compare our model with other kinetic descriptions of opinion dynamics present in the literature. In Section \ref{control} we propose a control strategy to reduce the number of inactive agents and increase the number of active ones. We determine under which circumstances the strategy is effective and we discuss its interpretation in terms of a real-world case scenario. In Section \ref{compactness} we determine some analytical properties of the models (both the uncontrolled model and the controlled one) and we prove several results of relaxation to equilibrium for their solutions. In particular, we separately investigate the evolutions of the activity level and of the opinion, in terms of their marginal distributions. Our main theoretical result of the paper is the proof of convergence to equilibrium for solutions to a Fokker--Planck equation with time-dependent coefficients, based on rigorous relative entropy estimates that allow us to determine the rate of relaxation. We conclude this work in Section \ref{concl}, by discussing some possible future developments of our research. 

\section{A kinetic model for opinion dynamics with variable activity level} \label{sec1}

    \noindent Let us consider a large population of indistinguishable agents, each characterized by their opinion $w$ and by their level of activity $A$, describing how active they are in seeking information and discussing the issue with their acquaintances. We assume that $A \in \R$ and $w \in I = [-1,1]$, where the values $w = \pm 1$ correspond to two opposite beliefs. The opinion formation process is characterized by two main mechanisms \cite{8}: 
    \begin{itemize}
        \item a compromise dynamics, i.e., individuals tend to settle their differences of opinion; \\[-3mm]
        \item an opinion fluctuation, i.e., every interaction between individuals is associated with a variation in opinions due to self-thinking. 
    \end{itemize}
    Similarly, the level of activity of an agent varies according to two processes \cite{1}:
    \begin{itemize}
        \item it increases when an agent is willing to interact; \\[-3mm]
        \item it gradually decreases over time. 
    \end{itemize}
    The evolution of the level of activity does not depend on the opinion dynamics, but the probability of considering the compromise dynamics of the opinion exchanges depends on the level of activity. Indeed, as an agent's level of activity increases, we expect that they will be more willing to share their opinion and to be influenced by the others' opinions. Each individual interacts with the other agents of the population and eventually also with opinion leaders \cite{2}. Opinion leaders are characterized only by their opinion since we assume that it is always possible to interact with them.  

    Let us start by modeling the agent--agent interactions. Consider an agent characterized by a level of activity and an opinion defined by the couple $(A,w)$. When they interact with another agent having activity level and opinion $(B,v)$, the probability that their opinion varies due to compromise dynamics depends on the respective levels of activity $A$ and $B$. Following \cite{1}, we introduce a discrete random variable $\Tilde{A}_p$ (prescribing when a change of opinion occurs) that assumes values in $\{0, 1\}$ with probability
    \begin{equation*} \label{eq1}
        P(\Tilde{A}_p=1)=\bar{A} \omega_p+\varepsilon,
    \end{equation*}
    where (see Figure \ref{f1}) 
    \begin{align} \label{eq2}
        \bar{A}= 
        \begin{cases}
            1 &\textnormal{if}   \; \; A \in \left[\gamma, +\infty\right), \\[2mm] 
            \frac{1}{2} + \frac{A}{ 2\gamma} & \textnormal{if}  \;\; A \in \left(-\gamma,  \gamma\right), \\[2mm] 
            0 & \textnormal{if} \;\;  A \in \left(-\infty, -\gamma\right], 
        \end{cases}
    \end{align}
    and $\omega_p,   \varepsilon,    \gamma > 0$ are given constants. The population is therefore divided into three parts: 
    \begin{itemize}
        \item the individuals with a level of activity larger than $\gamma$ are the active ones; \\[-3mm]
        \item the individuals with a level of activity lower than $- \gamma$ are the inactive ones; \\[-3mm]
        \item the individuals with a level of activity between $- \gamma$ and $ \gamma$ are the undecided ones, where we refer to the indecision regarding whether or not to participate in exchanges in social comparisons, not regarding their opinion about a certain topic. 
    \end{itemize}
    Due to the constant $\varepsilon$, inactive agents still have a small probability to interact. Clearly, $\omega_p + \varepsilon < 1$ and $\varepsilon < \omega_p$. The post-interaction opinion and level of activity $(A',w')$ are given by 
    \begin{align} \label{eq3}
    \begin{cases} 
        w'= w + \tilde{A}_p   \tilde{B}_p  G(w,v)   \lambda_p   (v-w) + D(w)   \eta_p, \\[2mm] 
        A'= A+\lambda_A(\tilde{A}_p - a_p),
    \end{cases}
    \end{align}
    where $\lambda_p, \lambda_A \in (0,1)$. The function $G(w,v) \in [0, 1]$ describes how strongly the opinion $v$ affects an agent with opinion $w$, therefore $G(w,v)=\bar{G}(|w-v|)$ where $\bar{G}$ is a decreasing function such that $\bar{G}(0)=1$. The opinion fluctuation is modeled through the centered random variable $\eta_p$ which has finite variance $\sigma_p^2>0$. If $G \equiv 1$, the constant $\lambda_p$ would measure the propensity to move toward the average opinion in the society (compromise dynamics), while the variance $\sigma_p^2$ measures the degree of spreading of opinion because of diffusion (self-thinking). The function $D(w) \in [0, 1]$ encodes the relevance of the diffusion, for example agents with an indifferent opinion ($w \simeq 0$) diffuse the most ($D(0)=1$), while those with an extreme opinion ($w \simeq \pm 1$) are less influenced by external factors ($D(\pm 1)=0$). Finally, $a_p \in (0,1)$ represents the gradual decrease of the activity level, which in turn can increase at each step according to $\tilde{A}_p$. 

    \begin{figure}[t]
    \includegraphics[width=8cm]{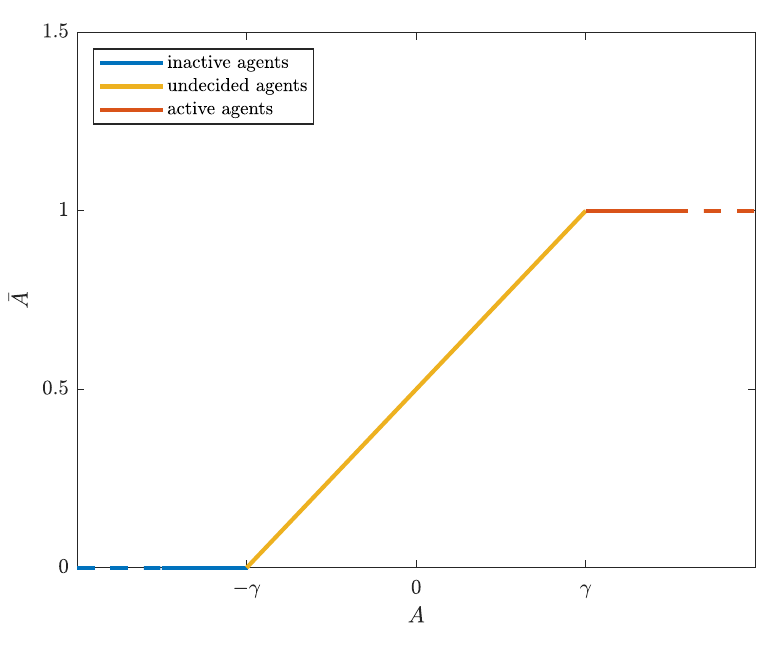}
    \caption{Graph of the function $\bar{A}$ defined by (\ref{eq2}). The three different parts of the population are highlighted.} \label{f1}
    \end{figure} 

    The interactions with a leader having opinion $z \in I$ are modeled in a similar way, the only difference is that only the activity level of the agent is considered since we assume that it is always possible to interact with them, namely
    \begin{align} \label{eq4}
    \begin{cases}
        w'=  w + \tilde{A}_\ell   G(w,z)   \lambda_\ell   (z-w) + D(w) \eta_\ell, \\[2mm] 
        A'= A+\lambda_A(\tilde{A}_\ell - a_\ell), 
    \end{cases}
    \end{align}
    where $\lambda_\ell, a_\ell \in (0,1)$ are given constants, $\eta_\ell$ is a centered random variable with finite variance $\sigma_\ell^2>0$, and $\Tilde{A}_l$ is a discrete random variable taking values $\{0, 1\}$ with probability
    \begin{equation*} \label{eq5}
        P(\tilde{A}_\ell=1)=\bar{A} \omega_\ell+\varepsilon,
    \end{equation*}
    where $\omega_\ell>0$, $\omega_\ell + \varepsilon < 1$, and $\varepsilon < \omega_\ell$. 

    Following \cite{8}, we may assume that 
    \begin{equation*}
        D(w)=\sqrt{1-w^2}, 
    \end{equation*}
    which ensures that the opinions $w'$ and $z'$ stay in the interval $I$. Note however that several other choices are possible \cite{6, 8}.

    The value of the parameter $a_\bullet$ leads to different situations. Fixed a level of activity $A$, we get 
    \begin{equation*} \label{eq8}
        \left< \tilde{A}_\bullet - a_\bullet \right> = \bar{A} \omega_\bullet + \varepsilon - a_\bullet,
    \end{equation*}
    where $\left<\cdot\right>$ denotes the mean with respect to the random variable. Therefore the following cases must be distinguished: 
    \begin{enumerate} [(i)]
        \item if $a_\bullet \in (0, \varepsilon)$ then, on average, the activity increases for any individual; \\[-3mm]
        \item if $a_\bullet \in (\varepsilon, \omega_\bullet + \varepsilon)$ then, on average, the activity of an inactive individual decreases while the activity of an active one increases; \\[-3mm]
        \item if $a_\bullet \in (\omega_\bullet + \varepsilon, +\infty)$ then, on average, the activity decreases for any individual. 
    \end{enumerate}
    A value of $a_\bullet$ can be chosen by fixing $M> \gamma$ and by imposing that 
    \begin{equation*} \label{eq9}
        \int_{-M}^M \left(  \bar{A} \omega_\bullet + \varepsilon - a_\bullet \right) \dd A = 0,
    \end{equation*}
    which leads to 
    \begin{equation} \label{eq10}
        a_\bullet = \frac{\omega_\bullet}{2} + \varepsilon. 
    \end{equation}
    This choice will ensure that (see Propositions \ref{prop3} and \ref{prop4}), if the agents are initially symmetrically distributed with respect to the level of activity $A$, the average level of activity of the population does not change over time. Notice that we fall into case (ii). 

    Consider now the binary interaction in (\ref{eq3}). Since $G$ is symmetric, we find 
    \begin{equation} \label{eq11}
    \begin{split}
        \left\langle w'+v'\right\rangle & = w + v + \tilde{A}_p   \tilde{B}_p   \lambda_p   (G(w,v) - G(v,w))     (v-w) \\[2mm] 
        & = w+v, 
    \end{split}
    \end{equation}
    therefore this (microscopic) binary interaction preserves the average opinion. On the other hand, the binary interaction in (\ref{eq4}), on average, brings the agent opinion closer to the one of leaders since they are not affected by the population opinion. 

    In the kinetic approach, the statistical description of the agents of the population is obtained through the distribution function $f=f(t,A,w),   t \in \R_+$. This means that  $f(t,A,w) \dd A \dd w$ represents the number of agents with opinion in $[w, w+\dd w]$ and level of activity in $[A, A+\dd A]$ at time $t \in \R_+$. Without loss of generality, we may assume that $\int_{\R\times I} f(0,A,w) \dd A \dd w = 1$. We will be interested in studying the evolution of the mass fractions of each part of the population, i.e., active, undecided and inactive, respectively given by
    \begin{align} 
        \rho_a(t) &= \int_{(\gamma, +\infty) \times I} f(t,A,w) \dd A \dd w, \nonumber \\[2mm]
        \rho_u(t) &= \int_{(-\gamma, \gamma) \times I} f(t,A,w) \dd A \dd w, \label{eq:mass fractions} \\[2mm]
        \rho_i(t) &= \int_{(-\infty, -\gamma) \times I} f(t,A,w) \dd A \dd w, \nonumber
    \end{align}
    and clearly $\sum_{j \in \{a,u,i\}}\rho_j(0)=1$. Moreover, we will also investigate the evolution of the average population opinion
    \begin{equation} \label{eq:m_w}
         m_w(t)=\frac{1}{\displaystyle \sum_{j \in \{a,u,i\}}\rho_j(t)}\int_{\R\times I} w f(t,A,w) \dd A \dd w, 
    \end{equation}
    and of the average level of activity of the population 
    \begin{equation} \label{eq:m_A}
         m_A(t)=\frac{1}{\displaystyle \sum_{j \in \{a,u,i\}}\rho_j(t)}\int_{\R\times I} A f(t,A,w) \dd A \dd w. 
    \end{equation}
    Finally, we shall denote with $f_\ell=f_\ell(z)$ the (static) opinion distribution of the leaders, satisfying $\int_I f_\ell(z)   \dd z = 1$, and with $\mu_\ell = \int_I z f_\ell(z)   \dd z$ their constant average opinion.  

    By classical methods of kinetic theory \cite{9} and resorting to the derivation of the classical linear Boltzmann equation for elastic rarefied gases \cite{10}, one can show that the time variation of the distribution function $f(t,A,w)$ depends on a sequence of elementary binary interactions of type (\ref{eq3}) and (\ref{eq4}). In our case, its evolution is given by the integro-differential equation
    \begin{equation*}
    \begin{split}
        \partial_t f(t,A,w) &= \left\langle \int_{\R \times I} \frac{1}{J(A',w')} \big(f(t,A',w') f(t,B',v') - f(t,A,w) f(t,B,v) \big) \dd B \dd v  \right\rangle \\[2mm]
        & \qquad + \left\langle \int_{I} \frac{1}{J(A',w')} \big(f(t,A',w') - f(t,A,w) \big) f_\ell(z) \dd z  \right\rangle,
    \end{split}
    \end{equation*}
    where $J(A',w')$ denotes the Jacobian of the transformation $(A',w') \mapsto (A,w)$. Here, the first operator translates the interactions agent--agent (\ref{eq3}), while the second one accounts for the interactions agent--leaders (\ref{eq4}). In particular, given some smooth functions $\varphi=\varphi(A,w)$ (the observable quantities), the above relation is better understood written in weak form
    \begin{equation} \label{eq25}
    \begin{split}
        \frac{\dd}{\dd t} \int_{\R\times I} \varphi(A,w) f(t,A,w) \dd A \dd w &= \int_{\R^2\times I^2}  \left<\varphi(A',w')-\varphi(A,w)\right> f(t,A,w) f(t,B,v)   \dd A   \dd B    \dd w   \dd v \\[2mm]
        & \qquad + \int_{\R\times I^2} \left<\varphi(A',w')-\varphi(A,w)\right> f(t,A,w) f_\ell(z) \dd A \dd w   \dd z.
    \end{split}
    \end{equation} 

    The choice $\varphi \equiv 1$ in (\ref{eq25}) leads to the conservation of mass, i.e., the total number of agents inside the population does not change. Therefore $\sum_{j \in \{a,u,i\}}\rho_j(t)=1$ for all $t \in \R_+$. Considering now the sole effects of the interactions agent--agent, since an alternative formulation of (\ref{eq25}) is given by
    \begin{equation} \label{eq26}
    \begin{split}
        \frac{\dd}{\dd t} \int_{\R\times I} \varphi(A,w) f(t,A,w) \dd A \dd w = \frac{1}{2}\int_{\R^2\times I^2} \big\langle\varphi(A',w') + & \varphi(B',v') - \varphi(A,w) - \varphi(B,v)\big\rangle \\[2mm] 
        & \times f(t,A,w) f(t,B,v)   \dd A   \dd B    \dd w   \dd v, 
    \end{split}
    \end{equation}
    the choice $\varphi(A,w)=w$ also shows the conservation of the average opinion $m_w$, thanks to (\ref{eq11}). 

\subsection{Derivation of a mean-field description} \label{subsec1} 

    In order to study the evolution of the macroscopic quantities previously introduced, following \cite{8} we perform a quasi-invariant limit procedure based on the standard rescalings
    \begin{equation*} \label{eq27}
        \lambda_\bullet \mapsto \epsilon \lambda_\bullet, \qquad \sigma_\bullet^2 \mapsto \epsilon \sigma_\bullet^2, \qquad t \mapsto \epsilon t,
    \end{equation*}
    where $0< \epsilon \ll 1$, and continuing to denote with $f(t,A,w)$ the corresponding rescaled distribution for the sake of simplicity.
    The above imply that the interactions (\ref{eq3}) and (\ref{eq4}) produce only an extremely small variation of the opinion and of the level of activity, while the frequency of interactions is increased accordingly. Our objective is to derive a Fokker--Planck-type equation that describes the evolution of the rescaled distribution $f(t,A,w)$. Since the calculations are cumbersome, for the sake of clarity we only sketch them in the case of agent--agent interactions (the interactions with the leaders can be treated in the exact same way). First, we Taylor expand the terms encoding the binary interaction in (\ref{eq3}), to recover
    \begin{equation*}
    \begin{split}
        \varphi(A',w') - \varphi&(A,w) = \partial_w \varphi (A,w) (w' - w) + \partial_A \varphi (A,w) (A'-A) \\[2mm] 
        & + \frac{1}{2} \partial^2_w \varphi (A,w) (w' - w)^2 + \partial^2_{wA} \varphi (A,w) (A'-A)  (w'-w) + \frac{1}{2} \partial^2_A \varphi (A,w) (A'-A) \\[4mm] 
        &+ R_\epsilon(A,A',w,w'), 
    \end{split}
    \end{equation*}
    where $R_\epsilon(A,A',w,w')$ denotes the remainder of the expansion. Hence, we can rewrite (\ref{eq25}) as 
    \begin{equation*}
    \begin{split}
        \frac{\dd}{\dd t} \int_{\R\times I} \varphi(A,w) & f(t,A,w) \dd A \dd w =\\[2mm]
         & \frac{1}{\epsilon} \int_{\R^2\times I^2} \partial_w \varphi (A,w) \left< w' - w \right> f(t,A,w) f(t,B,v)   \dd A   \dd B    \dd w   \dd v \\[2mm] 
        +& \frac{1}{\epsilon} \int_{\R^2\times I^2}  \partial_A \varphi (A,w) \left< A' - A \right> f(t,A,w) f(t,B,v)   \dd A   \dd B    \dd w   \dd v \\[2mm] 
        +& \frac{1}{2\epsilon} \int_{\R^2\times I^2}  \partial^2_w \varphi (A,w) \left<(w' - w)^2\right> f(t,A,w) f(t,B,v)   \dd A   \dd B    \dd w   \dd v \\[2mm] 
        +& \frac{1}{\epsilon} \int_{\R^2\times I^2}  \partial^2_{wA} \varphi (A,w) \left<(A'-A)  (w'-w)\right> f(t,A,w) f(t,B,v)   \dd A   \dd B    \dd w   \dd v \\[2mm] 
        +& \frac{1}{2\epsilon} \int_{\R^2\times I^2}  \partial^2_A \varphi (A,w) \left< (A' - A)^2\right> f(t,A,w) f(t,B,v)   \dd A   \dd B    \dd w   \dd v \\[2mm] 
        +& \frac{1}{\epsilon} \int_{\R^2\times I^2} R_\epsilon(A,A',w,w') f(t,A,w) f(t,B,v)   \dd A   \dd B    \dd w   \dd v. 
    \end{split}
    \end{equation*} 
    Next, we use (\ref{eq3}) to substitute the expressions for $w'$ and $A'$ and we compute the means of the random variables to cancel out some terms, recalling that $\left< \eta_p \right>=0$. Then, taking the limit $\epsilon \to 0$, additional terms cancel out if one assumes in particular $\left< |\eta_p|^3 \right> < +\infty$, so that \cite{11,8} 
    \begin{equation*}
        \frac{1}{\epsilon} \int_{\R^2\times I^2} R_\eps(A,A',w,w') f(t,A,w) f(t,B,v)   \dd A   \dd B    \dd w   \dd v \underset{\epsilon \to 0}{\longrightarrow} 0. 
    \end{equation*}
    Finally, assuming no-flux boundary conditions and integrating back by parts in order to get rid of the derivatives of $\varphi(A,w)$ (after repeating the same calculations for the interactions with the leaders), we formally obtain a Fokker--Planck-type equation describing the evolution of $f(t,A,w)$ \cite{11,9,15,8}, of the form
    \begin{equation} \label{eq29}
        \partial_t f(t,A,w)=Q_p(f,f)(t,A,w)+Q_\ell(f, f_\ell)(t,A,w), 
    \end{equation}
    where 
    \begin{equation} \label{eq30}
    \begin{split}
        Q_p(f,f)(t,A,w) &= \frac{\sigma_p^2}{2} \partial^2_w\Big(D^2(w)f(t,A,w)\Big) +\lambda_p(\bar{A}\omega_p+\varepsilon)\partial_w\Big(\mathcal{K}[f](t,w)f(t,A,w)\Big)\\[2mm] 
        & \qquad -\lambda_A \partial_A\Big((\bar{A}\omega_p+\varepsilon-a_p)f(t,A,w)\Big)
    \end{split}
    \end{equation}
    is the collisional operator related to the interactions agent--agent, while 
    \begin{equation} \label{eq31}
    \begin{split}
        Q_\ell(f, f_\ell)(t,A,w) &= \frac{\sigma_\ell^2}{2}\partial^2_w\Big(D^2(w)f(t,A,w)\Big) +\lambda_\ell(\bar{A}\omega_\ell+\varepsilon)\partial_w\Big(\mathcal{J}[f_\ell](w)f(t,A,w)\Big)\\[2mm] 
        & \qquad -\lambda_A\partial_A\Big((\bar{A}\omega_\ell+\varepsilon-a_\ell)f(t,A,w)\Big),
    \end{split}
    \end{equation}
    is the collisional operator related to the interactions agent--leaders. In particular, we have defined 
    \begin{align*}
        \mathcal{K}[f](t,w) &= \int_{\R\times I} (\bar{B}\omega_p+\varepsilon)G(w,v)(w-v) f(t,B,v)   \dd B   \dd v, \\[2mm]
        \mathcal{J}[f_\ell](w) &= \int_I G(w,z)(w-z) f_\ell(z)   \dd z.
    \end{align*}
    As already mentioned, we also have to complete the Fokker--Planck equation (\ref{eq29}) with no-flux boundary conditions for any $t \in \R_+$, specifically
    \begin{align} \label{eq35}
    \begin{cases}
        \left. D^2(w) f(t,A,w) \right|_{w=\pm 1} = 0, \quad & \forall  A \in \R, \\[4mm] 
        \left.(\bar{A}\omega_p+\varepsilon)\lambda_p\mathcal{K}[f](t,w)f(t,A,w)+\frac{\sigma_p^2}{2} \partial_w\Big(D^2(w)f(t,A,w)\Big)\right|_{w=\pm 1} = 0, & \forall A \in \R, \\[4mm] 
        \left.(\bar{A}\omega_\ell+\varepsilon)\lambda_\ell\mathcal{J}[f_\ell](w)f(t,A,w)+\frac{\sigma_\ell^2}{2} \partial_w\Big(D^2(w)f(t,A,w)\Big)\right|_{w=\pm 1} = 0, & \forall A \in \R, \\[4mm] 
        \displaystyle \lim_{A \to \pm \infty} A^\kappa f(t,A,w) = 0, & \forall w \in I, 
    \end{cases}
    \end{align}
    where $\kappa \geq 1$, and the last condition ensures that the first $\kappa$-th moments in $A$ of $f(t,A,w)$ are finite. We will prove (see Proposition \ref{prop5}) that, if the initial distribution has a compact support (as one would expect), then the support of $f(t,A,w)$ remains bounded for any time $t > 0$, hence the last condition of \eqref{eq35} is not restrictive. Note that the first condition of (\ref{eq35}) is related to the property $D(\pm 1)=0$. Finally, notice that the parameters $\lambda_p$ and $\lambda_\ell$, measuring the agents' disposition to compromise, appear in the drift terms of (\ref{eq30}) and (\ref{eq31}), while the parameters $\sigma_p^2$ and $\sigma_\ell^2$, measuring the agents' propensity to self-thinking, appear as coefficients in the diffusion operators. Consequently, small values of $\nu_p = \sigma_p^2/\lambda_p$ and $\nu_\ell = \sigma_\ell^2/\lambda_\ell$ correspond to compromise-dominated opinion dynamics, while large values of these ratios denote self-thinking-dominated opinion dynamics.

\subsection{Evolution of the macroscopic quantities} \label{subsec2}

    In this subsection we are interested in studying the behavior of the mass fractions \eqref{eq:mass fractions} and of the means \eqref{eq:m_w}--\eqref{eq:m_A} of the solution to the Fokker--Planck equation (\ref{eq29}). In fact, this simplified formulation of the Boltzmann description \eqref{eq25} will allow us to derive explicit information regarding their evolution. 
    
    Let us start by considering only the interactions agent--agent, namely $\partial_t f = Q_p(f,f)$. Simple calculations show that the no-flux boundary conditions and the symmetry of the function $G$ ensure the conservation of mass and of the average opinion $m_w$, already obtained through the integro-differential equations (\ref{eq25}) and (\ref{eq26}). However, since the Fokker--Planck model is less complex than its Boltzmann counterpart (\ref{eq25}), we are now also able to study the evolution of the mass fractions \eqref{eq:mass fractions} and of the average activity level \eqref{eq:m_A}. Since 
    \begin{align} 
        \frac{\dd}{\dd t} \rho_a(t) &= \lambda_A(\omega_p+\varepsilon-a_p) \int_I f(t, \gamma, w)  \dd w, \label{eq37} \\[2mm]
        \frac{\dd}{\dd t} \rho_i(t) &= -\lambda_A(\varepsilon-a_p) \int_I f(t,- \gamma, w)   \dd w, \label{eq38}
    \end{align}
    the different ranges of values to which $a_p$ belongs (cases (i)--(iii) described at the beginning of the section) lead to different situations. Recalling that $\frac{\dd}{\dd t} \rho_u = -\frac{\dd}{\dd t} \rho_a - \frac{\dd}{\dd t} \rho_i$, we distinguish between the following three cases: 
    \begin{itemize} 
        \item case (i): $\rho_a$ increases, $\rho_i$ decreases, and $\rho_u$ can show different behaviors depending on the initial distribution; \\[-3mm]
        \item case (ii): $\rho_a$ and $\rho_i$ increase, therefore $\rho_u$ decreases due to the conservation of the total mass; \\[-3mm]
        \item case (iii): $\rho_a$ decreases, $\rho_i$ increases, and $\rho_u$ can show different behaviors depending on the initial distribution. 
    \end{itemize}
    Recall that the choice of $a_p$ given by (\ref{eq10}) belongs to case (ii). Notice that the integrals on the right-hand sides of (\ref{eq37}) and (\ref{eq38}) can be interpreted as two fluxes at $A= \gamma$ and $A=- \gamma$, respectively.

    Consider now the evolution of the average level of activity inside the population. The fourth boundary condition of (\ref{eq35}) gives 
    \begin{equation*} \label{eq42}
        \frac{\dd}{\dd t} m_A(t) = \lambda_A \int_{\R\times I} (\bar{A}\omega_p+ \varepsilon-a_p) f(t,A,w) \dd A \dd w, 
    \end{equation*}
    therefore different choices of $a_p$ lead to these situations:
    \begin{itemize} 
        \item case (i): $m_A$ increases; \\[-3mm] 
        \item case (ii): we cannot conclude since the behavior depends on the initial distribution; \\[-3mm]
        \item case (iii): $m_A$ decreases. 
    \end{itemize}
    Finally, in a particular but still significant situation, the following proposition holds. 

    \begin{proposition} \label{prop3}
        Assume that $m_A(t)$ is analytical. Suppose that for any fixed $w \in I$ the initial distribution $f^\init(A,w) = f(w,A, 0)$ is even with respect to $A$. If $a_p$ is given by (\ref{eq10}), then $m_A \equiv 0$. 
    \end{proposition}
    \begin{proof}
        We want to show that $\left.\frac{\dd^n}{\dd t^n} m_A(t) \right|_{ t=0}=0$ for all $n \in \mathbb{N}$. The symmetry with respect to $A$ of $f^\init(A,w)$ ensures that $m_A(0)=0$ and
        \begin{equation*}
            \left.\frac{\dd}{\dd t} m_A(t) \right|_{ t=0} = \lambda_A \omega_p \int_{\R\times I} \left(\bar{A}-\frac{1}{2}\right) f^\init(A,w) \dd A \dd w = 0. 
        \end{equation*}
        A simple calculation shows that for all $n \geq 2$
        \begin{equation*}
            \left.\frac{\dd^n}{\dd t^n} m_A(t) \right|_{ t=0} = \frac{\lambda_A^n \omega_p^n}{ (2\gamma)^{n-1}} \int_{I \times \left(-\gamma, \gamma\right)} \left(\bar{A}-\frac{1}{2}\right) f^\init(A,w) \dd A \dd w = 0, 
        \end{equation*}
        therefore the result follows. 
    \end{proof}

    Consider now the effect of the leaders, namely $ \partial_t f = Q_p(f,f) + Q_\ell(f, f_\ell)$. In this situation, since the microscopic variations of $A$ in (\ref{eq3}) and (\ref{eq4}) have the same structure, the evolution of the activity level depends on $\omega_p+\omega_\ell$ and $a_p+a_\ell$, specifically 
    \begin{equation} \label{eq1000}
        \frac{\dd}{\dd t} m_A(t) = \lambda_A \int_{\R\times I} \left(\bar{A}(\omega_p+\omega_\ell)+ 2\varepsilon-(a_p+a_\ell)\right) f(t,A,w) \dd A \dd w. 
    \end{equation}
    In particular, a version of Proposition \ref{prop3} holds even in the presence of leaders. The important difference from the model without leaders is that now the mean opinion is not preserved since leaders tend to bring the agents' opinion closer to theirs, indeed
    \begin{equation*} \label{eq44}
        \frac{\dd}{\dd t} m_w(t) = -\lambda_\ell \int_{\R\times I^2} (\bar{A}\omega_\ell + \varepsilon) G(w,z) (w-z) f(t,A,w) f_\ell(z) \dd A \dd w   \dd z. 
    \end{equation*}
    For the specific case $G \equiv 1$, we get 
    \begin{equation} \label{eq45}
        \frac{\dd}{\dd t} m_w(t) = -\lambda_\ell \int_{\R\times I} (\bar{A}\omega_\ell + \varepsilon) (w-\mu_\ell) f(t,A,w) \dd A \dd w 
    \end{equation}
    and we conjecture the following. 
    
    \begin{conjecture} \label{prop2}
        Suppose that $G \equiv 1$, then the average opinion of the population $m_w(t)$ converges toward the average opinion of the leaders $\mu_\ell$, when $t \to +\infty$. 
    \end{conjecture}
    
    In Subsection \ref{evol2} we will prove Conjecture \ref{prop2} in a particular, but still realistic, case. 

\subsection{Opinion polarization and consensus formation} \label{pol_con}

    The Fokker--Planck equation (\ref{eq29}) does not admit a steady state, since the evolution of the activity level cannot reach an equilibrium (a rigorous proof of this fact will be presented in Section \ref{compactness}). Nonetheless, in order to understand how opinions are more likely to distribute over time inside the population, it is still possible to look for a partial equilibrium of the Fokker--Planck equation, namely for a distribution function which is not affected by the drift and diffusion operators prescribing the evolution in $w$. For the sake of clarity, we start by considering only the interactions agent--agent. The presence of opinion leaders will be discussed at the end of this subsection and will lead to analogous results. Moreover, suppose that $G \equiv 1$ and $D(w)=\sqrt{1-w^2}$. 
    
    For this, let us introduce the weighted (with weight $\bar{A} \omega_p + \varepsilon$) mass and average opinion of the population
    \begin{equation} \label{pesi}
    \begin{split}
        \rho_{\bar{A}}(t) &= \int_{\R\times I} (\bar{A}\omega_p+\varepsilon) f(t,A,w) \dd A \dd w, \\[2mm]
        m_{\bar{A}}(t) &= \int_{\R\times I} (\bar{A}\omega_p+\varepsilon) wf(t,A,w) \dd A \dd w,
        \end{split}
    \end{equation}
    which are well defined (i.e. they are finite) and not conserved over time. Then, the sought partial equilibria $f^{\eq, w}(t,A,w)$ are given by any distribution canceling out both terms of the Fokker--Planck operator $Q_p(f,f)$ that act on the opinion variable, i.e., satisfying the relation
    \begin{equation*}
        \lambda_p (\bar{A}\omega_p + \varepsilon)(w\rho_{\bar{A}}(t)-m_{\bar{A}}(t))f^{\eq,w}(t,A,w)+ \frac{\sigma^2_p}{2} \partial_w\Big((1-w^2)f^{\eq,w}(t,A,w)\Big)=0,
    \end{equation*}
    which, after computing the derivative with respect to $w$, leads to Beta distributions of the form
    \begin{equation} \label{eqbeta}
        f^{\eq, w}(t,A,w) = C(t,A) (1-w)^{-1+\frac{\bar{A}\omega_p+\varepsilon}{\nu_p} \big(\rho_{\bar{A}}(t)-m_{\bar{A}}(t)\big)}(1+w)^{-1+\frac{\bar{A}\omega_p+\varepsilon}{\nu_p} \big(\rho_{\bar{A}}(t)+m_{\bar{A}}(t)\big)},
    \end{equation}
    where $C(t,A)>0$ is a suitable normalization function ensuring that $\int_{\R\times I} f^{\eq, w}(t,A,w) \dd A \dd w = 1$ for any $ t > 0$. The partial equilibrium given by the Beta-type distributions (\ref{eqbeta}) can represent two different scenarios: 
    \begin{itemize}
        \item none of the exponents of the distribution are negative, thus $f^{\eq, w}(t,A,\pm 1)=0$ and, for fixed $t$ and $A$, the maximum of $f^{\eq, w}(t,A,w)$ is obtained for $w \in (-1,1)$. This situation corresponds to a consensus formation due to the compromise process of the binary interaction in (\ref{eq3}) \cite{3} (where full consensus would be associated to a Dirac delta distribution); \\[-3mm]
        \item at least one of the exponents is negative, hence the distribution diverges. This situation corresponds to an opinion polarization inside the population, and it is commonly driven by a self-thinking process that is stronger than the compromise propensity \cite{16}. 
    \end{itemize}
    Clearly the formation of (partial) consensus is preferable over opinion polarization, and in general one would like to prevent the formation of extreme opinions. For example, in \cite{16} it was shown that extreme opinions regarding a disease could trigger an increase in the spread of infection among individuals. Since the quantity $\bar{A}\omega_p+\varepsilon$ appears at the exponents of (\ref{eqbeta}), the scenario corresponding to active agents can be very different compared to the one where inactive agents are involved. In fact, noticing that since $|m_{\bar{A}}(t)| < \rho_{\bar{A}}(t)$ one has $\left(\rho_{\bar{A}}(t) \pm m_{\bar{A}}(t)\right)/\nu_p > 0$ (here, the equalities hold if and only if all agents have opinions $w=1$ or $w=-1$, which is the uninteresting case leading to the formation of Dirac delta distributions instead of the Beta distributions \eqref{eqbeta}), three possible situations must be distinguished: 
    \begin{itemize} 
        \item $\varepsilon \frac{\rho_{\bar{A}}(t) - m_{\bar{A}}(t)}{\nu_p} < (\omega_p + \varepsilon) \frac{\rho_{\bar{A}}(t) - m_{\bar{A}}(t)}{\nu_p} < 1$ and/or $\varepsilon \frac{\rho_{\bar{A}}(t) + m_{\bar{A}}(t)}{\nu_p} < (\omega_p + \varepsilon) \frac{\rho_{\bar{A}}(t) + m_{\bar{A}}(t)}{\nu_p} < 1$, which corresponds to opinion polarization among both active and inactive agents; 
        \item $1 < \varepsilon \frac{\rho_{\bar{A}}(t) \pm m_{\bar{A}}(t)}{\nu_p} < (\omega_p + \varepsilon) \frac{\rho_{\bar{A}}(t) \pm m_{\bar{A}}(t)}{\nu_p}$, which corresponds to consensus formation among both active and inactive agents; 
        \item $\varepsilon \frac{\rho_{\bar{A}}(t) - m_{\bar{A}}(t)}{\nu_p} < 1$ and/or $\varepsilon \frac{\rho_{\bar{A}}(t) + m_{\bar{A}}(t)}{\nu_p} < 1$, and  $1 < (\omega_p + \varepsilon) \frac{\rho_{\bar{A}}(t) \pm m_{\bar{A}}(t)}{\nu_p}$, which corresponds to opinion polarization among inactive agents and consensus formation among active agents.  
    \end{itemize}
    The third scenario is the most interesting one and highlights the importance of social interactions. Indeed, in this case, an effective compromise process among active agents prevents the formation of extreme opinion, while the absence of such process among inactive individuals leads to polarization. Moreover, there exists a unique threshold $A_*=A_*(t) \in (-  \gamma,  \gamma)$, given by   
    \begin{equation*}
        A_*(t) = \frac{1}{\omega_p}\min{\left(\frac{\rho_{\bar{A}}(t) - m_{\bar{A}}(t)}{\nu_p}, \frac{\rho_{\bar{A}}(t) + m_{\bar{A}}(t)}{\nu_p}\right)}^{-1} - \frac{\eps}{\omega_p}, 
    \end{equation*}
    which  distinguishes between opinion polarization and consensus formation (this situation is illustrated in Figure \ref{f2}). Notice that it is not possible to have consensus formation among inactive agents and, at the same time, opinion polarization among active ones, since the compromise dynamics are always stronger among active individuals. Finally, observe that as the value of $\nu_p$ decreases, i.e. the compromise dynamics become more relevant, the value $\left(\rho_{\bar{A}}(t) \pm m_{\bar{A}}(t)\right)/\nu_p$ increases, which highlights again the importance of compromise processes to prevent opinion polarizations. In particular, $\rho_{\bar{A}}(t) \le \omega_p + \varepsilon < 1$ and we thus need small values of $\nu_p$ to observe consensus formation (notice that the choice $\nu_p<1$ could still lead to polarized opinions, see Figure \ref{f2}). 

    \begin{figure}
    \centering
    \begin{subfigure}{0.475\textwidth}
        \centering
        \includegraphics[width=\textwidth]{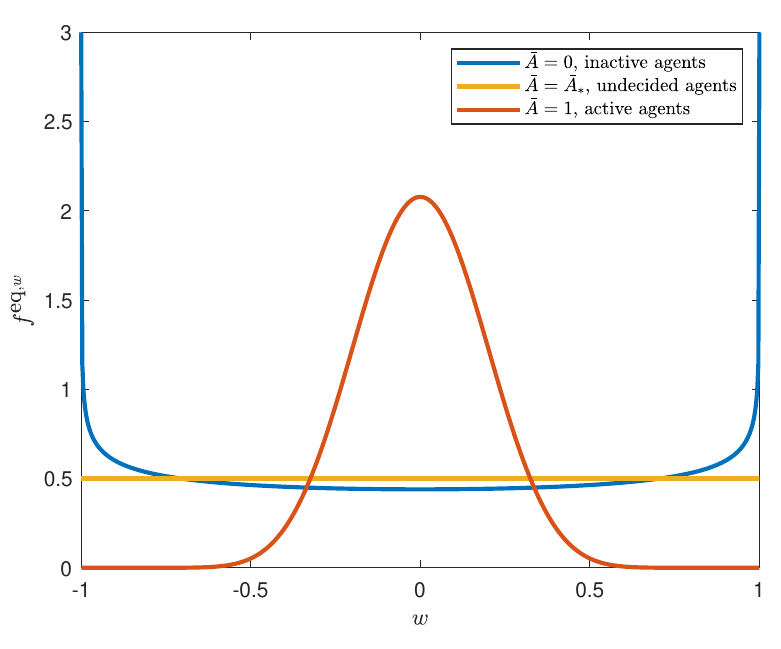}
        \caption{}
    \end{subfigure}
    \hspace{5mm}
    \begin{subfigure}{0.475\textwidth}
        \centering
        \includegraphics[width=\textwidth]{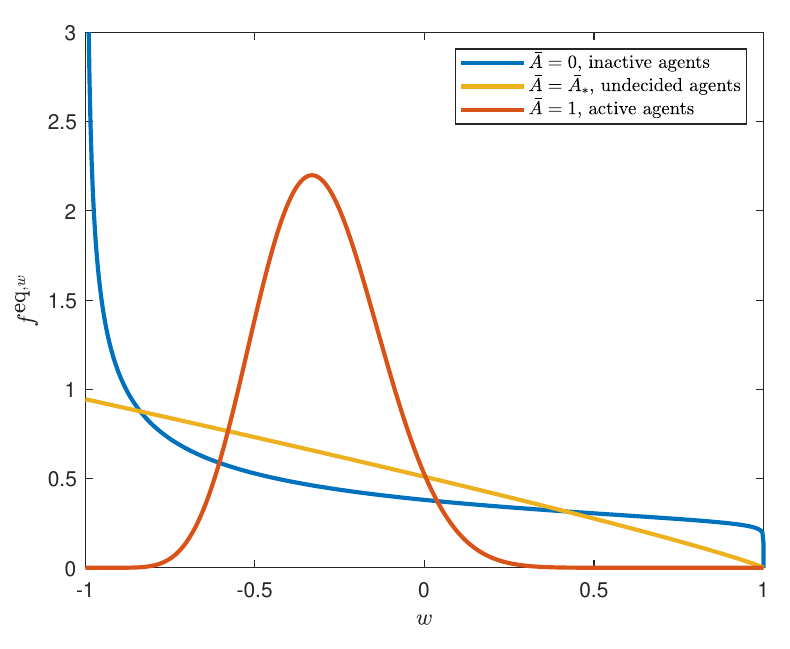}
        \caption{}
    \end{subfigure}
    \caption{Partial equilibrium distributions for a fixed time and different values of the activity level. Inactive agents are characterized by opinion polarization, while active ones by consensus formation. In (A) $m_{\bar{A}}=0$ while in (B) $m_{\bar{A}}=-0.2$, the other values are $\rho_{\bar{A}}=0.65$, $\omega_p=0.8$, $\varepsilon=0.05$, and $\nu_p=0.04$. The function $C$ was so chosen so that the area subtended by the graphs was equal to 1, but other choices would have been acceptable (and they would have led analogous results).} \label{f2}
    \end{figure} 

    Consider now the sole interactions with opinion leaders, namely only the collisional operator (\ref{eq31}). Analogous calculations show that the partial equilibrium in this case reads 
    \begin{equation} \label{eqbeta2}
        f^{\eq, w}(t,A,w) = C(t,A) (1-w)^{-1+\frac{\bar{A}\omega_\ell+\varepsilon}{\nu_\ell} (1-\mu_\ell)} (1+w)^{-1+\frac{\bar{A}\omega_\ell+\varepsilon}{\nu_\ell} (1+\mu_\ell)},
    \end{equation} 
    and the previous remarks hold also for these equilibria, since they have the same structure of (\ref{eqbeta}). Notice that \eqref{eqbeta2} is well-defined as long as $\mu_\ell \not= \pm 1$, since in such cases we would obtain a Dirac delta (in Subsection \ref{evol2} we will provide more details on this matter).

    At last, if we take into account both agent--agent and agent--leaders interactions, we end up with a similar Beta distribution combining the dependencies on $\rho_{\bar{A}}$, $m_{\bar{A}}$, and $\mu_\ell$, namely 
    \begin{equation}  \label{eqbeta3}
    \begin{split}
        f^{\eq, w}(t,A,w) &= C(t,A) (1-w)^{-1+\frac{\lambda_p(\bar{A}\omega_p+\varepsilon)(\rho_{\bar{A}}(t)-m_{\bar{A}}(t))+\lambda_\ell(\bar{A}\omega_\ell+\varepsilon)(1-\mu_\ell)}{\sigma_p^2 + \sigma_\ell^2}} \\
        & \hspace*{2cm} \times (1+w)^{-1+\frac{\lambda_p(\bar{A}\omega_p+\varepsilon)(\rho_{\bar{A}}(t)+m_{\bar{A}}(t))+\lambda_\ell(\bar{A}\omega_\ell+\varepsilon)(1+\mu_\ell)}{\sigma_p^2 + \sigma_\ell^2}}. 
    \end{split}
    \end{equation}
    All previous considerations still hold in this general case. 
    
\subsection{Comparison with other models} \label{subsec3}

    The model that we have introduced takes inspiration from the one in \cite{1} and can be seen as its kinetic counterpart. The most relevant difference consists in the way microscopic variations (\ref{eq3}) of the opinion work. Indeed, in \cite{1} the probability that an agent interacts with another solely depends on the activity level of the former, but when the interaction takes place both agents are influenced. Therefore, in order to reproduce the original model with a kinetic approach, we should consider interactions of the form
    \begin{align*}
    \begin{cases}
        w' =  w + \tilde{A}_{w,p}   G(w,v)   \lambda_p   (v-w) + \tilde{A}_{v,p}   G(w,v)   \lambda_p   (v-w) + D(w)   \eta_p, \\[2mm] 
        v'   =  v + \tilde{A}_{w,p}   G(v,w)   \lambda_p   (w-v) + \tilde{A}_{v,p}   G(v,w)   \lambda_p   (w-v) + D(v)   \eta_p, 
    \end{cases}
    \end{align*}
    where the discrete random variables $\tilde{A}_{w,p}$ and $\tilde{A}_{v,p}$ are related to the activity levels of the first and second agent, respectively. We believe that our approach to model these microscopic interactions is more realistic, since we consider the attitudes of both individuals. Another important difference resides in the fact that in \cite{1} it is supposed that the agents are distributed on a lattice, and that each of them is able to interact only with their neighbors. In order to reproduce this effect, we would have to assume that the frequency of agent--agent interactions depend on an underlying graph structure, which could be modeled through a graphon $\mathcal{B}: [0,1]^2 \to \R_+$ \cite{g1,g2,3, 5}. More precisely, one would have to further consider the static position $x \in [0,1]$ of each agent on a graphon structure of the form
    \begin{align*}
        \mathcal{B}(x,y)=
        \begin{cases}
            1 \quad & \textnormal{if}\ \  |x-y| \le d, \\[2mm] 
            0 & \textnormal{if} \ \   |x-y| > d,
        \end{cases}
    \end{align*}
    for some $d > 0$, and a model with similar properties to ours could then be derived by inserting $\mathcal{B}(x,y)$ as a coefficient multiplying the first integrand in (\ref{eq25}), see \cite{3}. 
  
    In recent years, several kinetic models for opinion dynamics have been proposed, see for example \cite{2, 3, 15, 6, 8, 16}. The main feature that distinguishes our model from these is clear: opinion variations depend not only on the relative differences between the pre-interaction opinions of agents and leaders, but also on the activity levels of the interacting individuals, measuring how much the latter ones are prone to search for new information and to modify their original opinion. In order to better understand this difference, let us introduce the marginal distribution 
    \begin{equation*} \label{eq46}
        h(t,w) = \int_{\R} f(t,A,w)   \dd A .
    \end{equation*}
    The Fokker--Planck equation (\ref{eq29}) for the distribution $h(t,w)$ reads 
    \begin{equation} \label{eq43}
        \partial_t h(t,w) = \tilde{Q}_p(h,f)(t,w)+\tilde{Q}_\ell(h, f_\ell)(t,w),
    \end{equation}
    where  
    \begin{align*}
        \tilde{Q}_p(h,f)(t,w) &= \frac{\sigma_p^2}{2} \partial^2_w\Big(D^2(w)h(t,w)\Big) +\lambda_p\partial_w\Big(\mathcal{K}[f](t,w)h_{\bar{A},p}(t,w)\Big), \\[2mm]
        \tilde{Q}_\ell(h, f_\ell)(t,w) &= \frac{\sigma_\ell^2}{2}\partial^2_w\Big(D^2(w)h(t,w)\Big) +\lambda_\ell\partial_w\Big(\mathcal{J}[f_\ell](w)h_{\bar{A},\ell}(t,w)\Big),
    \end{align*}
    and we have defined
    \begin{equation*} \label{eq49}
        h_{\bar{A},\bullet}(t,w) = \int_{\R} (\bar{A}\omega_\bullet + \varepsilon) f(t,A,w)   \dd A. 
    \end{equation*}
    While the Fokker--Planck equation (\ref{eq43}) displays a very similar structure to already existing models \cite{2,3,8,16}, we can see the additional influence of the activity level appearing in the distributions $h_{\bar{A},\bullet}(t,w)$ and in the operator $\mathcal{K}[f](t,w)$. The evolution of the marginal distribution $h(t,w)$ will be explored in more details in Section \ref{compactness}.

\section{Increasing the number of active agents via control strategies} \label{control}

    \noindent In this section, we introduce a suitable control strategy to increase the overall activity level of the population. In the political sphere, this control strategy could be interpreted as a way to prevent abstentionism \cite{12, 14, 13}. Our approach resembles the one introduced in \cite{3} to avoid consensus formation on a graphon: we include an additional microscopic dynamic to prevent the formation of a large group of inactive agents. In particular, we aim at: 
    \begin{itemize}
        \item allowing the inactive agents to quickly increase their activity level; \\[-3mm]
        \item avoiding the presence of agents with an extremely high or low level of activity $A$, a scenario that would be unrealistic and difficult to interpret. 
    \end{itemize}
    To this end, we modify the activity levels of the individuals based on a convex combination of updates weighted by a parameter $\theta \in (0,1)$, so that a fraction $1-\theta$ of the population updates its level of activity based on the second relations of (\ref{eq3}) and (\ref{eq4}), while the other fraction of size $\theta$ does it following the controlled interaction 
    \begin{equation} \label{eq:controlled interaction}
        A'=A-\lambda_c A, 
    \end{equation}
    where $\lambda_c \in (0,1)$ characterizes the intensity of the control. Notice that this controlled interaction becomes more relevant as $|A|$ grows. The idea is to bring negative values of the activity level closer to zero, so that inactive agents can first become undecided and later on move to the class of active agents. At the same time, we need to ensure that the strategy does not reduce the number of active agents (realistically, think for example of an agent that suddenly loses interest regarding a particular topic). We will illustrate a way to decrease the number of inactive agents while simultaneously increasing the number of active ones. We also stress that interaction \eqref{eq:controlled interaction} is designed to avoid the formation of agents with extremely high or low levels of activity. The proof of this fact is more technical and will be presented in Section \ref{compactness}, together with a comparison between the controlled and the uncontrolled models. 

    \begin{remark}
        The control strategy that we consider is based on the idea that an agent could suddenly gain or lose interest in a certain topic, resulting in an increase or in a reduction of their social interactions. Hence, it can be seen as a microscopic update of the activity level which is alternative to the original one given by \eqref{eq3}. Such strategy could be reformulated in a more canonical way (see, e.g., \cite{3}) by interpreting the controlled interaction \eqref{eq:controlled interaction} as a microscopic update of the form
        \begin{equation*}
            A' = A + \alpha A_c,
        \end{equation*}
        where $\alpha > 0$, and the optimal value $A^{\textrm{opt}}$ of the control $A_c$ is obtained by solving the minimization problem
        \begin{equation*}
            A_{\mathrm{opt}} = \underset{A_c \in \mathcal{A}_c}{\arg \min} \left( (A' - A_{\mathrm{target}})^2 + \beta A_c^2 \right),
        \end{equation*}
        on the set $\mathcal{A}_c$ of admissible controls. Here, $A_{\mathrm{target}} \in \R$ denotes a prescribed target value and $\beta > 0$ is a suitable regularization parameter encoding the magnitude of the cost needed to apply the control strategy. By the method of Lagrange multipliers, the above constrained problem transforms into finding the unique solution $A_c \in \R$ of
        \begin{equation*}
            (A' - A_{\mathrm{target}}) \frac{\partial A'}{\partial A_c} + \beta A_c = 0,
        \end{equation*}
        which is given by
        \begin{equation*}
            A_{\mathrm{opt}} = - \frac{\alpha}{\alpha^2 + \beta} (A - A_{\mathrm{target}}).
        \end{equation*}
        In particular, we end up with the optimal microscopic update
        \begin{equation*}
            A' = A - \frac{\alpha^2}{\alpha^2 + \beta}(A - A_{\mathrm{target}}).
        \end{equation*}
        Now, we have mentioned that the focus of our control strategy is to promote the agents' activity, by bringing negative values of the activity level closer to zero. Therefore, it is natural to choose as target value $A_{\mathrm{target}} = 0$. Then, the interaction \eqref{eq:controlled interaction} is recovered by assuming that the intensity of the control and the magnitude of the cost needed to apply it are equivalent, namely $\alpha = \beta$, and by taking $\alpha = \frac{\lambda_c}{1-\lambda_c}$, which is indeed well defined as long as $\lambda_c \in (0,1)$.
    \end{remark}

    Since the interactions agent--agent and agent--leaders modify the activity level according to the same kind of process, for the sake of clarity we shall consider in this section only agent--agent exchanges inside the population. The presence of opinion leaders would lead to analogous results, but notice that some calculations would slightly change due to the fact that the evolution of the activity level would depend on $\omega_p + \omega_\ell$ and $a_p+a_\ell$ (see Eq. (\ref{eq1000})). In particular, it would be enough to replace $a_p$ with $a_p + a_\ell$, $\omega_p$ with $\omega_p + \omega_\ell$, and $\varepsilon$ with $2\varepsilon$. 

\subsection{Mean-field description of the controlled model} \label{sec:mean}

    Proceeding like in Subsection \ref{subsec1}, we introduce the rescaling
    \begin{equation*} \label{eq52}
        \lambda_c \mapsto \epsilon\lambda_c,
    \end{equation*}
    and we take the limit $\epsilon \to 0$ (see Subsection \ref{subsec1}) to recover the Fokker--Planck-type equation 
    \begin{equation} \label{eq53}
        \partial_t f(t,A,w)=(1-\theta) Q_p(f,f)(t,A,w)+\theta Q_c(f,f)(t,A,w),
    \end{equation}
    with
    \begin{equation} \label{eq54}
    \begin{split}
        Q_c(f,f)(t,A,w) &= \frac{\sigma_p^2}{2}  \partial^2_w\Big(D^2(w)f(t,A,w)\Big) +\lambda_p(\bar{A}\omega_p+\varepsilon)\partial_w\Big(\mathcal{K}[f](t,w)f(t,A,w)\Big)\\[2mm] 
        & \qquad + \lambda_c \partial_A\Big(Af(t,A,w)\Big). 
    \end{split}
    \end{equation}
    Notice that the only difference between the original collisional operator $Q_p(f,f)$, defined by (\ref{eq30}), and the new collisional operator $Q_c(f,f)$ derived from the control strategy and defined by (\ref{eq54}), lies in the drift term acting on $A$ (recall that the control strategy solely influences the microscopic update of the activity level). The Fokker--Planck equation (\ref{eq53}) is then completed with the no-flux boundary conditions (\ref{eq35}), where we take $\kappa \geq 2$. Similarly to what done before, in Corollary \ref{equibnd} we will prove that if the initial distribution has a compact support, then the support of $f(t,A,w)$ remains equi-bounded for any $t > 0$, and thus the choice of the boundary conditions is not restrictive. This also justifies the fact that in the derivation of the collisional operator $Q_c(f,f)$ given by \eqref{eq54}, it is possible to neglect terms involving the factor $\epsilon A$, even though $A$ itself is not bounded. For these reasons, from now on we shall always assume that the initial distribution has compact support. 

    \begin{remark}
        The partial equilibrium of the controlled model shares the same structure of the one determined for the original model (\ref{eqbeta}). The only difference is that the quantities $\rho_{\bar{A}}$ and $m_{\bar{A}}$ evolve in a different way, hence the evolution of the normalization function $C(t,A)$ changes accordingly. 
    \end{remark}

\subsection{Effects of the control strategy on the agents' activity level} 

    We start by limiting ourselves to the analysis of case (ii) of Section \ref{sec1}, i.e., $a_p \in (\varepsilon, \omega_p+\varepsilon)$, since it is the most realistic and interesting one. Similarly to the uncontrolled model, the total mass and the average opinion of the population are conserved. Concerning the mass fractions of active and inactive agents, we instead now get 
    \begin{equation*}
    \begin{split}
        \frac{\dd}{\dd t} \rho_i (t) &= -\overbrace{\left((1-\theta)\lambda_A(\varepsilon-a_p)+\theta\gamma \lambda_c \right)}^{\mathcal{C}_i} \int_I f(t,- \gamma, w)   \dd w, \\[2mm]
        \frac{\dd}{\dd t} \rho_a (t) &= \overbrace{\left((1-\theta)\lambda_A(\omega_p+\varepsilon-a_p)-\theta\gamma \lambda_c \right)}^{\mathcal{C}_a} \int_I f(t,\gamma, w)   \dd w. 
    \end{split}
    \end{equation*}
    Therefore, in order to increase the number of active individuals while reducing the number of inactive ones, we need to impose that the coefficients $\mathcal{C}_i$ and $\mathcal{C}_a$ are positive, thus the parameter $\lambda_c$ has to satisfy 
    \begin{equation} \label{eq58}
        \frac{1-\theta}{\gamma \theta} \lambda_A (a_p-\varepsilon) < \lambda_c < \frac{1-\theta}{\gamma \theta} \lambda_A (\omega_p + \varepsilon - a_p),  
    \end{equation}
    and the set of the solutions is not empty if an only if 
    \begin{equation} \label{eq59}
        a_p \in \left(\varepsilon, \frac{\omega_p}{2}+\varepsilon\right) \subset (\varepsilon, \omega_p+\varepsilon). 
    \end{equation}
    If we fix $\lambda_c=\lambda_A$ to deal with a lower number of free parameters, $\theta$ must then satisfy 
    \begin{equation} \label{eq60}
        0 < \frac{a_p-\varepsilon}{\gamma+a_p-\varepsilon} < \theta < \frac{\omega_p + \varepsilon-a_p}{\gamma+\omega_p + \varepsilon-a_p} < 1, 
    \end{equation}
    and the set of the solutions is once again not empty if and only if \eqref{eq59} holds.
    Eq. (\ref{eq60}) tells us that a reduction in the number of inactive agents can be achieved when a sufficiently high number of individuals follows the control strategy, and that, at the same time, if too many agents follow this strategy then also the number of active individuals would decrease. 
    Moreover, we also deduce that the mass fraction $\rho_u$ of undecided agents does not show a monotone behavior since inactive agents close to $A=- \gamma$ tend to become undecided ($\mathcal{C}_i > 0$), while undecided agents close to $A= \gamma$ tend to become active ($\mathcal{C}_a > 0$). 
    
\subsection{Special cases} 

    We conclude the discussion of the control strategy by analyzing several special cases that we encountered. In particular, we would like to understand what happens if (\ref{eq59}) is not satisfied, and if $a_p$ is given by (\ref{eq10}). Additionally, we are also interested in studying cases (i) and (iii), when $a_p \in (0, \varepsilon)$ and $a_p \in (\omega_p + \varepsilon, +\infty)$, respectively. 

    Let us start by supposing that $a_p \in \left(\frac{\omega_p}{2}+\varepsilon, \omega_p + \varepsilon\right)$. This means that the activity level fades away faster than before and thus our control strategy is not effective, since only one inequality in (\ref{eq58}) can be satisfied. If we try to reduce the number of inactive agents ($\mathcal{C}_i > 0$), then the number of active ones decreases ($\mathcal{C}_a < 0$). Similarly, if $\lambda_c=\lambda_A$, only one inequality of (\ref{eq60}) can be satisfied and we end up with the same situation as before. 

    Suppose now that $a_p=\frac{\omega_p}{2}+\varepsilon$. By choosing 
    \begin{equation*}
        \lambda_c= \frac{1-\theta}{\theta}  \frac{\lambda_A\omega_p}{2\gamma}, 
    \end{equation*}
    or, if $\lambda_c = \lambda_A$, 
    \begin{equation*}
        \theta=\frac{\omega_p}{2\gamma+\omega_p},
    \end{equation*}
    we get $\mathcal{C}_i = \mathcal{C}_a = 0$, i.e. the fluxes cancel out. Therefore, the mass fractions of inactive and active agents remain constant. Clearly, in this case it is also possible to fix $\lambda_c$ (or $\theta$) such that $\mathcal{C}_i$ and $\mathcal{C}_a$ have opposite signs. With this choice, the considerations we have made for $a_p \in \left(\frac{\omega_p}{2}+\varepsilon, \omega_p + \varepsilon\right)$ are valid. Finally, Proposition \ref{prop3} holds even if we adopt the control strategy (even without making any hypotheses on $\lambda_c$ or $\theta$). 
    \begin{proposition} \label{prop4}
        Assume that $m_A(t)$ is analytical. Suppose that for any fixed $w \in I$ the initial distribution $f^\init(A,w) = f(0,A,w)$ is even with respect to $A$ and that the control strategy is implemented. If $a_p=\frac{\omega_p}{2}+\varepsilon$, then $m_A \equiv 0$.  
    \end{proposition}
    \begin{proof}
        It suffices to consider 
        \begin{equation*}
            \frac{\dd}{\dd t} m_A(t) = (1-\theta) \lambda_A \omega_p \int_{\R\times I} \left(\bar{A}-\frac{1}{2}\right)f(t,A,w) \dd A \dd w -\theta \lambda_c \int_{\R\times I} Af(t,A,w)  \dd A   \dd w,
        \end{equation*}
        and adapt the calculations from the proof of Proposition \ref{prop3}. 
    \end{proof}

    Consider now the case $a_p \in (0, \varepsilon)$. In this situation we do not need the control strategy to reduce the number of inactive agents (see Subsection \ref{subsec2}). Notice that we always have $\mathcal{C}_i > 0$, hence the number of inactive agents decreases. The introduction of the control strategy would conduct to two possible scenarios: if $\mathcal{C}_a < 0$, i.e. the strategy has a strong effect, then the number of active agents decreases; if $\mathcal{C}_a > 0$, i.e. the strategy has a weak effect, then the number of active agents increases. Therefore, in this case introducing the control strategy could only lead to a worse situation. 

    Finally, suppose that $a_p \in (\omega_p + \varepsilon, +\infty)$, which represents the worst case scenario (see Subsection \ref{subsec2}). Since $\mathcal{C}_a < 0$, the number of active agents decreases. Again, only two situations can be observed: if $\mathcal{C}_i > 0$, i.e. the effect of the control strategy is strong, then the number of inactive agents decreases because they become undecided; if $\mathcal{C}_i < 0$, i.e. the effect of the strategy is weak, then the number of inactive agents increases. Therefore, it is clearly better to introduce a strong control strategy in this case, but we cannot hope to obtain results as good as the ones found for $a_p \in (\varepsilon, \omega_p + \varepsilon)$, since the number of active agents always decreases. 

\section{Analytical properties and trends to equilibrium} \label{compactness}

    \noindent This final section is dedicated to prove some analytical properties and convergence to equilibrium results for the proposed models.

    In Subsection \ref{evol1}, we investigate the evolution of the social activity alone (independently of the opinion dynamics) by studying the marginal of $f(t,A,w)$ with respect to $A$. The idea here is that when the distribution of the activity level reaches some sort of equilibrium, studying the evolution of opinions becomes much easier. We are going to prove that, when considering the controlled model, all agents become active in finite time, allowing to consider reduced Fokker--Planck models governing the dynamics of the marginal of $f(t,A,w)$ with respect to $w$. Since the agents update their activity level by interacting with each other and with the leaders following the same structural rules, in Subsection \ref{evol1} it will be enough to consider the sole interactions agent--agent, as interactions agent--leaders could be treated in the same fashion. We will start from the uncontrolled model and then move on to the controlled one. In particular, we will show that the latter prevents the creation of agents with extremely high or low levels of activity (matching our second goal explained at the beginning of Section \ref{control}).
    
     In Subsection \ref{evol2}, we then move on to analyze the evolution of opinions alone, both interactions agent--agent and agent--leaders are taken into account. In particular, when considering the controlled model, all agents have become active and we will prove that the marginal of $f(t,A,w)$ in $w$ converges toward a global equilibrium state (a Beta-type distribution) of the corresponding reduced model, when $t \to +\infty$. Our analysis makes use of relative entropy techniques inspired from \cite{20,FurPulTerTos} and extended in a novel way to deal with Fokker--Planck equations that involve time-dependent coefficients.
     
     Finally, in Subsection \ref{evol3} we will derive some relevant analytical properties of the whole distribution $f(t,A,w)$, notably its nonnegativity, its $L^\infty(\R_+,L^1(\R \times I))$ regularity, and its uniqueness in a specific linearized setting. 

\subsection{Evolution of the social activity of individuals} \label{evol1}

    Let us introduce the marginal distribution 
    \begin{equation*} \label{eq66}
        g(t,A) = \int_I f(t,A,w)   \dd w,
    \end{equation*}
    and consider the Fokker--Planck equation \eqref{eq29} with only the operator $Q_p(f,f)$. Let us also assume that the initial level of activity inside the population is bounded, i.e., the initial distribution $f^\init(A,w) = f(0,A,w)$ has a compact support. For any value of the constant $a_p$ from cases (i), (ii), or (iii), since the activity level varies according to the drift term of (\ref{eq30}) acting on $A$ alone, one expects that the support of the distribution remains compact for any $t > 0$. We make this argument rigorous in the following proposition. 

    \begin{proposition} \label{prop5}
        Let us consider the uncontrolled model \eqref{eq29} with only agent--agent interactions. Given an initial distribution $g^\init(A) = g(0,A)$ such that $\textnormal{supp}( g^\init) \subset [i_1,i_2]$, then $\textnormal{supp}( g( t, \cdot)) \subset [i_1 + \lambda_A (\varepsilon - a_p)  t, i_2 + \lambda_A (\omega_p + \varepsilon - a_p)  t]$ for any $t \in \R_+$.
    \end{proposition}
    \begin{proof}
        Integrating over $w \in I$ equation \eqref{eq29} with only the operator $Q_p(f,f)$, the evolution of the distribution $g(t,A)$ is given by 
        \begin{equation} \label{eq65}
        \begin{split}
            \partial_t g(t,A) &= -\lambda_A \partial_A\Big((\bar{A}\omega_p + \varepsilon - a_p)  g(t,A)\Big) \\[2mm] 
            &= -\lambda_A (\bar{A}\omega_p + \varepsilon - a_p) \partial_A g(t,A) - \lambda_A \omega_p g(t,A) \partial_A \bar{A},
        \end{split}
        \end{equation}
        and can be read along the characteristic curves. By the method of characteristics, we look for the solutions $( t(s),A(s), g(s,A(s)))$ of
        \begin{equation} \label{eq:zero_char}
            \begin{array}{ll}
                \displaystyle \frac{\dd}{\dd s} t(s) = 1, &  t(0) = 0, \\[6mm]
                \displaystyle \frac{\dd}{\dd s} A(s)= \lambda_A (\bar{A}(s)\omega_p + \varepsilon - a_p), & A(0) = A_0 \in [i_1, i_2], \\[6mm]
                \displaystyle \frac{\dd }{\dd s} g(s,A(s)) = - \lambda_A \omega_p g(s,A(s)) \partial_A \bar{A}(s), &  g(0,A(0)) =  g^\init(A_0). 
            \end{array}
        \end{equation}
        We find $t=s$, then
        \begin{equation*}
            \frac{\dd}{\dd s} A(s) \in [\lambda_A (\varepsilon - a_p), \lambda_A (\omega_p + \varepsilon - a_p)],  
        \end{equation*}
        and from the third equation of \eqref{eq:zero_char} we deduce 
        \begin{equation} \label{exponential}
             g(s,A(s))= g(0,A_0) \exp \left(- \lambda_A \omega_p \int_0^s \partial_A \bar{A}(u)   \dd u\right). 
        \end{equation}
        The thesis follows from the fact that the minimum and maximum velocities describing the evolution of the activity level are therefore given by $\lambda_A (\varepsilon - a_p)$ and $\lambda_A (\omega_p + \varepsilon - a_p)$, respectively. 
    \end{proof} 

    In particular, if $a_p$ belongs to case (i) then $\displaystyle \lim_{t \to +\infty} \rho_A(t) = 1$, while if $a_p$ belongs to case (iii) then $\displaystyle \lim_{t \to +\infty} \rho_i(t) = 1$. On the other hand, case (ii) is more complex and interesting. Let us define $A_p^* \in (- \gamma,  \gamma)$ to be the unique solution to $\bar{A}\omega_p + \varepsilon - a_p = 0$. Agents with an initial activity level equal to $A_p^*$ do not change it since, assuming that $A(t)$ is analytical, recursive calculations easily show that $\left. \frac{\dd^n }{\dd  t^n} A(t) \right|_{t=0} = 0$ for all $n \geq 1$ (see the proof of Lemma \ref{lemma1}, which presents similar calculations). Consider, for example, an undecided agent with initial activity level $A_0$ greater than $A_p^*$. Then 
    \begin{align*}
    \frac{\dd }{\dd t} A(t)=
    \begin{cases}
            \lambda_A \left(\frac{\omega_p}{2\gamma} A(t) + \frac{\omega_p}{2} + \varepsilon - a_p \right) > 0  \quad & \textnormal{if }   A(t) < \gamma, \\[4mm] 
            \lambda_A (\omega_p + \varepsilon - a_p)>0  \quad & \textnormal{if }   A(t) \geq \gamma,
    \end{cases} 
    \end{align*}
    and their activity level initially increases exponentially before becoming an active agent in a finite time, from which point their activity level starts to increase linearly. The same reasoning would hold if $A_0 < A_p^*$, but in this case the level of activity would decrease and the agent would become inactive. This behavior was anticipated by (\ref{eq37}) and (\ref{eq38}). Notice that, unless all agents have an activity level equal to $A_p^*$ initially (i.e., $ g^\init(A)$ is given by the Dirac delta $\delta(A-A_p^*)$, which implies that $g(t,A) = \delta(A-A_p^*)$ for any $t \in \R_+$), the supports $\textnormal{supp}(g(t,\cdot))$, $t \in \R_+$, are not equi-bounded anymore. Therefore, there will be agents with extremely high or low levels of activity and the equilibrium $\delta(A-A_p^*)$ is unstable. 

    \begin{remark}
        The transport equation (\ref{eq65}) preserves the positivity of the initial datum, since the exponential function \eqref{exponential} is always positive. Moreover, because the total mass is conserved, the $L^1(\R)$ regularity of the initial datum is also preserved. 
    \end{remark}

    The controlled interaction (\ref{eq:controlled interaction}) is expected to prevent this unwanted situation characterized by extreme levels of activity since, as $|A|$ grows, the influence of the control increases. Suppose that $\lambda_c$ and $a_p$ satisfy (\ref{eq58}) and (\ref{eq59}) respectively, i.e., that the control strategy is useful and effective (similar results can be obtained in the other, less interesting, cases). Then, the evolution of the distribution $g(t,A)$ is prescribed by 
    \begin{equation} \label{eq67}
    \begin{split}
        \partial_t  g(t,A) & = -(1-\theta)\lambda_A \partial_A \Big((\bar{A}\omega_p + \varepsilon - a_p)  g(t,A)\Big) + \theta \lambda_c \partial_A \Big(A g(t,A)\Big) \\[2mm] 
        & = \partial_A \Big( \left( \theta \lambda_c A - (1-\theta) \lambda_A (\bar{A}\omega_p + \varepsilon - a_p) \right)  g(t,A) \Big).  
    \end{split}
    \end{equation}
    A simple calculation shows that there exists a unique solution to $\theta \lambda_c A - (1-\theta) \lambda_A (\bar{A}\omega_p + \varepsilon - a_p) = 0$, which is given by 
    \begin{equation} \label{eq68}
        A_c^* = \frac{1-\theta}{\theta} \frac{\lambda_A}{\lambda_c} (\omega_p + \varepsilon - a_p), 
    \end{equation}
    and satisfies $A_c^* >  \gamma$, due to (\ref{eq58}). Therefore, equation (\ref{eq67}) admits the equilibrium $g^{\infty}(A)=\delta(A-A_c^*)$, which intuitively seems to be stable thanks to the control strategy. Recall that also (\ref{eq65}) admitted a Dirac delta as equilibrium (case (ii)), but the latter was unstable. In the situation considered here, the following result holds. 
    
    \begin{proposition} \label{teo1}
        Consider the controlled model and suppose that (\ref{eq58}) and (\ref{eq59}) hold. Then, the characteristic curves $A(t)$ of (\ref{eq67}) converge exponentially toward $A_c^*$ given by (\ref{eq68}). 
    \end{proposition}

    \begin{proof} 
        The evolution of the distribution $g$ is given by 
        \begin{equation*} 
        \begin{split}
            \partial_t  g(t,A) = \left(\theta \lambda_cA - (1-\theta) \lambda_A (\bar{A}\omega_p + \varepsilon - a_p)\right) \partial_A g(t,A) + \left(\theta \lambda_c- (1-\theta)\lambda_A \omega_p \partial_A \bar{A} \right) g(t,A),
        \end{split}
        \end{equation*}
        and can be read along the characteristic curves, which are the solutions to 
        \begin{equation} \label{eq:characteristic}
            \begin{array}{ll}
                \displaystyle \frac{\dd}{\dd s} t(s) = 1, &  t(0) = 0, \\[6mm]
                \displaystyle \frac{\dd}{\dd s} A(s)= -\theta \lambda_cA(s) + (1-\theta) \lambda_A (\bar{A}(s)\omega_p + \varepsilon - a_p), & A(0) = A_0 \in \R, \\[6mm]
                \displaystyle \frac{\dd }{\dd s} g(s,A(s)) = \partial_A \left(\theta \lambda_cA(s) - (1-\theta) \lambda_A (\bar{A}(s)\omega_p + \varepsilon - a_p) \right) & \\
                \qquad\qquad\qquad\qquad \times g(s,A(s)), &  g(0,A(0)) =  g^\init(A_0). 
            \end{array}
        \end{equation}
        Hence $t=s$ and we distinguish between three different cases. 

        \medskip
        \noindent \textbf{Case 1:} $\mathbf{A_0 \geq \bm{\gamma}}.$ As long as $A(t) \geq  \gamma$, we have that
        \begin{equation} \label{eq:evolution of A}
            A(t) = A_c^* + (A_0 - A_c^*) \exp\left(-\theta \lambda_c  t\right), 
        \end{equation}
        from which we infer $A(t)> \gamma$ for any $t > 0$ and $\displaystyle \lim_{t \to +\infty} A(t) = A_c^*$. 

        \medskip
        \noindent \textbf{Case 2:} $\mathbf{A_0 \in (-\bm{ \gamma}, \bm{\gamma})}.$ As long as $A(t) \in (- \gamma,  \gamma)$ we have 
        \begin{equation*}
            A(t)=-\frac{(1-\theta)\lambda_A \left(\frac{\omega_p}{2} + \varepsilon - a_p\right)}{(1-\theta) \lambda_A \frac{\omega_p}{ 2\gamma} - \theta \lambda_c} +  \left(A_0 + \overbrace{\frac{(1-\theta)\lambda_A \left(\frac{\omega_p}{2} + \varepsilon - a_p\right)}{(1-\theta) \lambda_A \frac{\omega_p}{ 2\gamma} - \theta \lambda_c}}^{\mathcal{T}_1 }\right) \exp\left(\overbrace{\left((1-\theta) \lambda_A \frac{\omega_p}{2\gamma} - \theta \lambda_c\right)}^{\mathcal{T}_2} t\right). 
        \end{equation*}
        Since (\ref{eq58}) and (\ref{eq59}) imply that $\mathcal{T}_2 < 0$ and $\mathcal{T}_1 < - \gamma$, we deduce that $A_0 + \mathcal{T}_1 < 0$ and that $A(t)$ grows toward $-\mathcal{T}_1 >  \gamma$. Therefore $A(t)$ reaches $ \gamma$ in a finite time and then we fall into Case 1. 

        \medskip
        \noindent \textbf{Case 3:} $\mathbf{A_0 \le -\bm{ \gamma}}.$ As long as $A(t) \le - \gamma$ we have
        \begin{equation*}
            A(t)=-\frac{(1-\theta)\lambda_A(a_p-\varepsilon)}{\theta \lambda_c} +  \left(A_0 + \overbrace{\frac{(1-\theta)\lambda_A(a_p-\varepsilon)}{\theta \lambda_c}}^{\mathcal{T}_3} \right) \exp\left(-\theta \lambda_c  t\right).
        \end{equation*}
        Since (\ref{eq58}) and (\ref{eq59}) imply that $0 < \mathcal{T}_3 <  \gamma $, we find that $A_0 + \mathcal{T}_3 < 0$ and $A(t)$ grows toward $-\mathcal{T}_3 \in (- \gamma, 0)$. Therefore, after a finite time we fall back into Case 2. 
    \end{proof}

    The following corollary is obvious and represents a stronger version of Proposition \ref{prop5}, valid only for the controlled model. 

    \begin{corollary} \label{equibnd}
        Consider the controlled model and suppose that the initial distribution $ g^\init(A)$ has compact support, then the supports $\textnormal{supp}( g(t,\cdot))$ are equi-bounded for any $t \in \R_+$. 
    \end{corollary} 

    In particular, this result implies that the controlled model satisfies both requirements made at the beginning of Section \ref{control}. Additionally, we deduce the following convergence to equilibrium. 

    \begin{corollary} \label{cor2}
        Suppose that the initial distribution $ g^\init(A)$ has compact support. Then, the solutions to the transport equation (\ref{eq67}) converge in the sense of distributions toward the equilibrium $ g^{\infty}(A)=\delta(A-A_c^*)$. 
    \end{corollary}
    \begin{proof}
        We want to show that for any $\varphi \in C_c^\infty(\R)$ it holds
        \begin{equation*}
            \lim_{t \to +\infty} \int_\R \varphi(A)    g(t,A)   \dd A = \varphi(A_c^*). 
        \end{equation*}
        Let us first determine the explicit solution of equation \eqref{eq67} along the characteristic curves. By the method of characteristics, we look for solutions $(t(s),A(s), g(s,A(s)))$ of \eqref{eq:characteristic}. From the first equation we recover $ t = s$. In order to solve the second and third ones, we suppose that $A(s) > \gamma$ (thanks to Proposition \ref{teo1}, this assumption does not entail any loss of generality). Then, the conditions on $A(s)$ provide
        \begin{equation*}
            A(s) = A(0) e^{-\theta \lambda_c s} + A_c^* \left( 1 - e^{-\theta \lambda_c s} \right),
        \end{equation*}
        which can be inverted using the initial condition $A(0) = A_0$, to get
        \begin{equation} \label{eq:A}
            A_0 = A(s) e^{\theta \lambda_c s} - A_c^* \left( e^{\theta \lambda_c s} - 1 \right).
        \end{equation}
        Moreover, since $A(s) > \gamma$ for any $s$, from the third equation of \eqref{eq:characteristic} we deduce 
        \begin{equation*}
             g(s,A(s)) =  g(0,A_0) e^{\theta \lambda_c s}.
        \end{equation*}
        Combining the initial condition $g(0,A_0) =  g^\init(A_0)$ and the inverse relation \eqref{eq:A}, and going back to the original variables $(t,A)$, we finally conclude that the solution to \eqref{eq67} writes
        \begin{equation} \label{eq:solution of transport}
              g(t,A) =  g^\init\big( A e^{\theta \lambda_c t} - A^*_c \left( e^{\theta \lambda_c t} - 1 \right) \big) e^{\theta \lambda_c  t}.
        \end{equation}
        Now, Proposition \ref{teo1} ensures that the characteristics $A(t)$ fall into the interval $\left( \gamma, +\infty \right)$ in finite time, where we observe exponential convergence toward $A_c^*$ with explicit rate \eqref{eq:evolution of A} and we can apply the previous reasoning. Let us thus fix $T > 0$ as the minimal time such that $A(T) > \gamma$. Then, for any $ t \geq T$, the solution $  g(t,A)$ to \eqref{eq67} is given by \eqref{eq:solution of transport}. Therefore, we can successively compute
        \begin{equation*}
            \begin{split}
                \lim_{t \to +\infty} \int_{\R} \varphi(A) g(t,A) \dd A & = \lim_{\substack{t \to \infty \\[1mm]  t \geq T}} \int_{A > \gamma} \varphi(A)  g^\init\left( A e^{\theta \lambda_c t} - A^*_c \left( e^{\theta \lambda_c t} - 1 \right) \right) e^{\theta \lambda_c  t} \dd A \\[2mm]
                &= \lim_{\substack{t \to +\infty \\[1mm]  t \geq T}} \int_{B > A_c^* - \left( A_c^* - \gamma \right) e^{\theta \lambda_c t}} \varphi\left( B e^{-\theta \lambda_c t} + A^*_c \left( 1 - e^{-\theta \lambda_c t} \right) \right)  \\
                & \qquad \times g^\init(B) \dd B \\[2mm]
                &= \varphi(A_c^*) \int_{\R}  g^\init(B) \dd B \\[4mm] 
                & = \varphi(A_c^*),
            \end{split}
        \end{equation*}
        recovering the desired convergence of $  g(t,A)$ toward $\delta(A - A_c^*)$, in the limit $t \to +\infty$.
    \end{proof}

    Since $A_c^*> \gamma$, Corollary \ref{cor2} implies that all agents tend to become active. 

    \begin{remark}
        The transport equation (\ref{eq67}) preserves the positivity of the initial datum and hence also its $L^1(\R)$ regularity, thanks to mass preservation. Note that it cannot preserve the $L^p(\R)$ regularity for $p >1$, since the solutions converge toward a Dirac delta. 
    \end{remark}

\subsection{Opinion dynamics and the influence of leaders} \label{evol2}

    Going back to the original distribution $f(t,A,w)$ and still focusing on the controlled model, together with the assumptions $G \equiv 1$ and $D(w)=\sqrt{1-w^2}$, if the initial distribution $f^\init(A,w)$ has a compact support, Proposition \ref{teo1} ensures that all agents become active in finite time. Therefore, starting from such finite time, the mass of $f(t,A,w)$ with respect to $A$ is concentrated in the interval $( \gamma,+\infty)$. Integrating the Fokker--Planck equation \eqref{eq53} for $A \in ( \gamma,+\infty)$, we then infer that the marginal distribution
    \begin{equation*}
        h(t,w) = \int_\R f(t,A,w)   \dd A,
    \end{equation*}
    introduced in Subsection \ref{subsec3}, is solution to
    \begin{equation} \label{eqh1}
    \begin{split}
        \partial_t h(t,w) &= \frac{\sigma_p^2}{2} \partial^2_w\Big((1-w^2)h(t,w)\Big) +\lambda_p\partial_w\left(\mathcal{K}[f](t,w)\int_\R (\bar{A}\omega_p+\varepsilon) f(t,A,w)   \dd A \right) \\[2mm] 
        &=\frac{\sigma_p^2}{2} \partial^2_w\Big((1-w^2)h(t,w)\Big) + \lambda_p (\omega_p + \varepsilon)^2 \partial_w\Big((w-m_w^\init) h(t,w)\Big),  
    \end{split}
    \end{equation} 
    where $m_w^\init = m_w(0)$ is the initial average opinion of the population, which is preserved over time since we are considering only agent--agent interactions. It is well-known that the adjoint equation of \eqref{eqh1} corresponds to a Wright--Fisher-type model with constant coefficients, whose analytical properties have been investigated in several works, see e.g. \cite{CheStr,EpsMaz1,EpsMaz2}. In particular, its solution $h(t,w)$ belongs to $C^\infty(\R_+ \times (-1,1))$ and, close to the boundaries of $I$, behaves like the equilibrium of \eqref{eqh1} given by
    \begin{equation} \label{h_inf}
        h^\infty (w) = C   (1-w)^{-1+\frac{(\omega_p+\varepsilon)^2}{\nu_p} (1-m_w^\init)} (1+w)^{-1+\frac{(\omega_p+\varepsilon)^2}{\nu_p} (1+m_w^\init)},
    \end{equation}
    where $C > 0$ is a suitable normalization constant such that $\int_I h^\infty (w)   \dd w = 1$. Moreover, it is possible to prove that $h(t,w)$ converges in $L^1(I)$ at a rate which is at least $\smallO\left(t^{-1/2}\right)$ toward this equilibrium $h^\infty(w)$ \cite{FurPulTerTos,MR4020534}\footnote[1]{We highlight that the rate of convergence of $\smallO\left(t^{-1/2}\right)$ is valid \cite{FurPulTerTos} for any range of the parameters $\omega_p$, $\eps$, $\nu_p$, and $m_w^\init \in (-1,1)$, but the authors in \cite{MR4020534} were actually able to prove exponential relaxation toward the equilibrium \eqref{h_inf} as soon as these parameters satisfy the specific constraints $1 - \frac{\nu_p}{2(\omega_p+\eps)^2} > 0$ if $m_w^\init = 0$, and $1 - \frac{\nu_p}{2(\omega_p+\eps)^2} \geq |m_w^\init|$ otherwise. Note in particular that the rates of convergence of $h(t,w)$ toward $h^\infty(w)$ correspond to those obtained in \cite{FurPulTerTos,MR4020534}, because we proved (see Proposition \ref{teo1}) that the agents tend to become active exponentially fast.}. Note that the exponents of (\ref{h_inf}) depend explicitly on the quantities (\ref{pesi}), computed for a distribution $f(t,A,w)$ for which all agents are active. We thus deduce the following theorem. 

    \begin{theorem}
        Let us suppose that $G \equiv 1$, $D(w)=\sqrt{1-w^2}$, and the initial distribution $f^\init(A,w)$ has a compact support. If $f(t,A,w)$ evolves according to the controlled model (\ref{eq53}), then the marginal distribution $h(t,w)$ converges in $L^1(I)$ toward the equilibrium distribution $h^\infty(w)$ given by Eq. (\ref{h_inf}), at a rate which is at least $ \smallO\left(t^{-1/2}\right)$. 
    \end{theorem}
    
    Observe that the influence of the initial condition $f^\init(A,w)$ is only seen in the equilibrium $h^\infty(w)$ through the preserved initial average opinion $m_w^\init$. Moreover, notice that if we consider the original model and if $a_p$ belongs to case (i), then all agents tend to become active (linearly fast, see Proposition \ref{prop5}) and, in this situation, the evolution of $h(t,w)$ is again given by (\ref{eqh1}). Therefore, the marginal $h(t,w)$ converges toward the equilibrium $h^\infty(w)$. Obviously, an analogous situation occurs if $a_p$ belongs to case (iii), with $\varepsilon$ appearing at the exponents of (\ref{h_inf}) instead of $\omega_p + \varepsilon$.

    Consider now the presence of opinion leaders and recall that (see Eq. (\ref{eq1000})) the evolution of the activity level depends on the quantities $\omega_p + \omega_\ell$, $a_p + a_\ell$, and $2\varepsilon$, defined at the beginning of Section \ref{control}. Thanks to the results obtained in the previous Subsection \ref{evol1}, we are now able to demonstrate Conjecture \ref{prop2} (i.e., the relaxation of the population's average opinion $m_w(t)$ toward the leaders' average opinion $\mu_\ell$) under the realistic assumption that $f^\init(A,w)$ has a compact support. 

    \begin{proposition} \label{prop6}
        Let us consider the original kinetic model (\ref{eq29}) in the presence of opinion leaders, with $G \equiv 1$, and suppose that the initial distribution $f^\init(A,w)$ has a compact support. Then, the average opinion of the population $m_w(t)$ converges toward the average opinion of the leaders $\mu_\ell$, in the limit $t \to +\infty$. 
    \end{proposition}

    \begin{proof}
        Since the interactions agent--agent preserve the average opinion (see Subsection \ref{subsec2}), we can reduce our study to the sole interactions agent--leader. As pointed out at the beginning of Subsection \ref{evol1}, only three scenarios are possible: 
        \begin{itemize}
            \item if $a_p + a_\ell$ belongs to case (i), then for large times all agents become active; \\[-3mm]
            \item if $a_p + a_\ell$ belongs to case (iii), then for large times all agents become inactive; \\[-3mm]
            \item if $a_p + a_\ell$ belongs to case (ii), then for large times there are only active agents and inactive agents (under the realistic assumption that the initial number of agents with an activity level $A^* \in (- \gamma,  \gamma)$, defined as the unique solution to $\bar{A}(\omega_p + \omega_\ell) + 2\varepsilon - a_p - a_\ell = 0$, is negligible). 
        \end{itemize}
        In particular, in all three cases the mass fractions \eqref{eq:mass fractions} of the population are definitely constant, namely $\rho_i(t) = \bar{\rho}_i$, $\rho_u(t) = \bar{\rho}_u$ and $\rho_a(t) = \bar{\rho}_a$ for $t$ large enough. 

        If $a_p+a_\ell$ belongs to cases (i) or (iii), for large times we find that (see Eq. (\ref{eq45})) 
        \begin{equation*}
            \frac{\dd}{\dd t} m_w(t) = - \lambda_\ell (\omega_\ell + \varepsilon) (m_w(t) - \mu_\ell),
        \end{equation*}
        and 
        \begin{equation*}
            \frac{\dd}{\dd t} m_w(t) = - \lambda_\ell   \varepsilon (m_w(t) - \mu_\ell), 
        \end{equation*}
        respectively, thus $m_w(t)$ readily converges to $\mu_\ell$ when $t \to +\infty$. 

        Suppose now that $a_p+a_\ell$ belongs to case (ii). By defining the average opinion in the classes of active and inactive agents as
        \begin{equation*}
            \begin{split}
                m_a(t) &= \frac{1}{\rho_a(t)}\int_{(\gamma, +\infty) \times I} w f(t,A,w) \dd A  \dd w, \\[4mm]
                m_i(t) &= \frac{1}{\rho_i(t)}\int_{(-\infty, -\gamma) \times I} w f(t,A,w) \dd A  \dd w,
            \end{split}
        \end{equation*}
        and recalling that $\sum_{j \in \{a,u,i\}}\rho_j \equiv 1$ from the conservation of mass, we find that the average population's opinion $m_w(t)$ given by \eqref{eq:m_w} simplifies, for large times, to
        \begin{equation*}
            m_w(t) = \bar{\rho}_i m_i(t) + \bar{\rho}_a m_a(t). 
        \end{equation*} 
        If we can show that each $m_\bullet(t) \underset{t \to +\infty}{\longrightarrow} \mu_\ell$, then the thesis would follow from the fact that $\bar{\rho}_i + \bar{\rho}_a = 1$. For $m_i(t)$, since $\rho_i(t)$ is definitely constant, we find that 
        \begin{equation*}
        \begin{split}
            \frac{\dd}{\dd t} & m_i(t) = \frac{1}{\bar{\rho}_i}  \frac{\dd}{\dd t} \int_{ (-\infty, -\gamma) \times I} w   f(t,A,w) \dd A \dd w \\[2mm] 
            & = - \frac{1}{\bar{\rho}_i} \lambda_\ell   \varepsilon \int_{ (-\infty, -\gamma) \times I} (w-\mu_\ell) f(t,A,w) \dd A \dd w - \frac{1}{\bar{\rho}_i} \lambda_A (2\eps - a_p - a_\ell) \int_I w f(t,A,w) \Big|_{A \to -\infty}^{A=- \gamma}   \dd w \\[4mm] 
            & = -\lambda_\ell   \varepsilon (m_i(t) - \mu_\ell), 
        \end{split}
        \end{equation*}
        since there are no agents with an activity level equal to $- \gamma$ and due to the boundary conditions \eqref{eq35}. Hence $m_i(t) \underset{t \to +\infty}{\longrightarrow} \mu_\ell$. The proof for $m_a(t)$ is similar and the thesis follows. 
    \end{proof} 

    Corollary \ref{cor2} implies that Proposition \ref{prop6} holds even if we are considering the controlled model coupled with opinion leaders, since all agents tend to become active, and in this case $A_c^*$ \eqref{eq68} would be defined as $A_c^*= \frac{1-\theta}{\theta} \frac{\lambda_A}{\lambda_c} (\omega_p + \omega_\ell + 2\eps - a_p - a_\ell)$. 
    
    Concerning the convergence to equilibrium for the marginal distribution $h(t,w)$ evolving under the controlled model, two situations must be distinguished. If we consider only the interactions with the leaders, then $h(t,w)$ evolves according to
    \begin{equation*} \label{eqh1_bis}
    \begin{split}
        \partial_t h(t,w) &= \frac{\sigma_\ell^2}{2} \partial^2_w\Big((1-w^2)h(t,w)\Big) +\lambda_\ell\partial_w\left(\mathcal{J}[f_\ell](w)\int_\R (\bar{A}\omega_\ell+\varepsilon) f(t,A,w)   \dd A \right) \\[2mm] 
        &=\frac{\sigma_\ell^2}{2} \partial^2_w\Big((1-w^2)h(t,w)\Big) + \lambda_\ell (\omega_\ell + \varepsilon) \partial_w\Big((w-\mu_\ell) h(t,w)\Big),  
    \end{split}
    \end{equation*} 
    and it converges again in $L^1(I)$ at a rate which is at least $\smallO\left(t^{-1/2}\right)$ toward the function \cite{FurPulTerTos,MR4020534} (recall Eq. \eqref{eqbeta2})
    \begin{equation*} \label{h_inf_bis}
        h^\infty (w) = C (1-w)^{-1 + \frac{\omega_\ell+\varepsilon}{\nu_\ell} (1-\mu_\ell)}(1+w)^{-1 + \frac{\omega_\ell+\varepsilon}{\nu_\ell} (1+\mu_\ell)},
    \end{equation*}
    where $C>0$ is a suitable normalization constant\footnote[2]{Recall that this global equilibrium is well defined as long as $\mu_\ell \not= \pm 1$. In those cases $h(t,w)$ would converge (in the sense of distributions) toward the Dirac delta $\delta(w-\mu_\ell)$. Indeed, assuming for example $\mu_\ell=1$, for any $\epsilon > 0$ 
    \begin{equation*}
        \int_{-1}^{1-\epsilon} h(t,w)   \dd w \le  \frac{1}{\epsilon} \int_{-1}^{1-\epsilon} (1-w) h(t,w)   \dd w \le \frac{1-m_w(t)}{\epsilon} \underset{t \to +\infty}{\longrightarrow} 0
    \end{equation*}
    since $m_w(t) \underset{t \to +\infty}{\longrightarrow} 1$.}. On the other hand, if we consider both interactions among the agents and with the leaders, the distribution $h(t,w)$ satisfies 
    \begin{equation} \label{eq100}
    \begin{split}
        \partial_t h(t,w) &= \frac{\sigma_p^2}{2} \partial^2_w\Big((1-w^2)h(t,w)\Big) +\lambda_p   (\omega_p + \varepsilon)^2 \partial_w\Big((w-m_w(t)) h(t,w)\Big) \\[2mm] 
        & \qquad + \frac{\sigma_\ell^2}{2} \partial^2_w\Big((1-w^2)h(t,w)\Big) +\lambda_\ell   (\omega_\ell + \varepsilon) \partial_w\Big((w-\mu_\ell) h(t,w)\Big), 
    \end{split}
    \end{equation}
    with no-flux boundary conditions
    \begin{equation} \label{eq:main BC}
            \begin{cases}
                \frac{\sigma_p^2+\sigma_\ell^2}{2} \partial_w\Big((1-w^2)h(t,w)\Big) + \lambda_p (\omega_p + \varepsilon)^2 (w-m_w(t)) h(t,w) + \lambda_\ell (\omega_\ell + \varepsilon)^2 (w-\mu_\ell) h(t,w)\Big|_{w=\pm 1} = 0,  \\[4mm] 
                (1-w^2) h(t,w) \Big|_{w=\pm 1} = 0. 
            \end{cases}
        \end{equation}
    In particular, it is easy to check that the mean opinion $m_w(t)$ of $h(t,w)$ evolves according to
    \begin{equation} \label{eq:evolution of mean}
        \frac{\dd}{\dd t} (m_w(t) - \mu_\ell) = -\lambda_\ell(\omega_\ell+\varepsilon) (m_w(t) - \mu_\ell),
    \end{equation}
    and thus converges exponentially to the average opinion of the leaders $\mu_\ell$, when $t \to +\infty$. Therefore, we expect $h(t,w)$ to relax toward the global equilibrium state (see Eq. \eqref{eqbeta3})
    \begin{equation} \label{eq101}
        h^\infty (w) = C^\infty (1-w)^{-1+\frac{\lambda_p(\omega_p+\varepsilon)^2+\lambda_\ell(\omega_\ell+\varepsilon)}{\sigma_p^2 + \sigma_\ell^2} (1-\mu_\ell)} (1+w)^{-1+\frac{\lambda_p(\omega_p+\varepsilon)^2+\lambda_\ell(\omega_\ell+\varepsilon)}{\sigma_p^2 + \sigma_\ell^2} (1+\mu_\ell)},
    \end{equation}
    where $C^\infty > 0$ is a normalization constant explicitly given by $C^\infty = 1/B(b_1^\infty,b_2^\infty)$, where $B$ denotes the Beta function
    \begin{equation*}
        B(b_1^\infty,b_2^\infty) = \int_I (1-w)^{-1+b_1^\infty} (1+w)^{-1+b_2^\infty} \dd w,
    \end{equation*}
    of parameters $b_1^\infty$ and $b_2^\infty$ given by
    \begin{equation*}
    \begin{split}
        b_1^\infty &= \frac{\lambda_p(\omega_p+\varepsilon)^2+\lambda_\ell(\omega_\ell+\varepsilon)}{\sigma_p^2 + \sigma_\ell^2} (1-\mu_\ell), \\[2mm]
        b_2^\infty &= \frac{\lambda_p(\omega_p+\varepsilon)^2+\lambda_\ell(\omega_\ell+\varepsilon)}{\sigma_p^2 + \sigma_\ell^2} (1+\mu_\ell).
    \end{split}
    \end{equation*}
    Notice that when opinion leaders are taken into account, the influence of the initial distribution $f^\init(A,w)$ is not seen in the equilibrium $h^\infty(w)$. Moreover, like before, similar considerations obviously hold if all agents become (in)active under the evolution of the original model.
    
    Proving the above convergence turns out to be rather problematic, due to the appearance of the time-dependent quantity $m_w(t)$ inside the Fokker--Planck equation \cite{20}. In order to deal with it, we shall require additional regularity properties on the distribution $h(t,w)$ satisfying \eqref{eq100}.

    As shown in \cite[Proposition 1]{2}, Eq. \eqref{eq100} preserves the nonnegativity of any nonnegative initial datum $h^\init \in L^1(I)$, and hence also its $L^1(I)$ regularity, since the equation is mass-preserving. In Proposition \ref{nonneg} we will actually use a more general argument to prove that both the uncontrolled model \eqref{eq29} and the controlled one \eqref{eq53} preserve the nonnegativity of any nonnegative initial datum $f^\init \in L^1(\R \times I)$ for the whole distribution $f(t,A,w)$.
    
    On the other hand, the $L^q(I)$ regularity for $q>1$ is not always preserved. Indeed, two situations must be distinguished. If at least one of the exponents of \eqref{eq101} is negative (which corresponds to opinion polarization, see Subsection \ref{pol_con}), then the resulting Beta distribution belongs to the space $L^q(I)$ for any $1 \leq q < q^*$, where
    \begin{equation} \label{eq:Lq global equilibrium}
        q^* = \left( 1 - \frac{\lambda_p(\omega_p+\varepsilon)^2+\lambda_\ell(\omega_\ell+\varepsilon)}{\sigma_p^2 + \sigma_\ell^2} (1-|\mu_\ell|) \right)^{-1}. 
    \end{equation}
    On the other hand, if both exponents are nonnegative (which corresponds to consensus formation), then the equilibrium state \eqref{eq101} belongs to any $L^q(I)$ with $q \geq 1$ (in particular, it is in $L^\infty(I)$).
    
    The following result of convergence to equilibrium for $h(t,w)$ is obtained under some \textit{a priori} assumptions on the regularity of the marginal distribution, in particular the condition that the latter remains $L^q(I)$ regular, for some $q > 1$. Unfortunately, this assumption is not obvious to demonstrate and we will only be able to recover it in the case of consensus formation. Our proof closely follows the strategy of \cite{20}, based on the relative entropy method and developed to tackle the convergence to equilibrium for a nonlocal Fokker--Planck equation with time-dependent coefficients (the main difference there being the structure of the equilibria, which are given by Gaussian densities instead of our Beta distributions).

    \begin{theorem} \label{teo:conv}
        Let $h(t,w)$ be a solution of the Fokker--Planck equation \eqref{eq100} with initial datum $h^\init \in L^q(I)$, for some $q > 1$. If $h \in L^\infty(\R_+, L^q(I))$, then it converges in $L^1(I)$ toward the equilibrium state $h^\infty(w)$ given by \eqref{eq101} when $t \to +\infty$, with a rate $\smallO\left(t^{-1/2}\right)$, namely
        \begin{equation*}
            \norm{h-h^\infty}_{L^1(I)} = \smallO\left(t^{-1/2}\right), \quad t \to +\infty.
        \end{equation*}
    \end{theorem}

    \begin{proof}
        Let us denote with
        \begin{equation*}
        \begin{split}
            Q_p(h)(t,w) &= \frac{\sigma_p^2}{2} \partial^2_w\Big((1-w^2)h(t,w)\Big) +\lambda_p   (\omega_p + \varepsilon)^2 \partial_w\Big((w-m_w(t)) h(t,w)\Big), \\[2mm] 
            Q_\ell(h)(t,w) &= \frac{\sigma_\ell^2}{2} \partial^2_w\Big((1-w^2)h(t,w)\Big) +\lambda_\ell   (\omega_\ell + \varepsilon) \partial_w\Big((w-\mu_\ell) h(t,w)\Big),
        \end{split}
        \end{equation*}
        the Fokker--Planck operators appearing on the right-hand side of \eqref{eq100}.
        
        Following \cite{20}, the main idea of the proof is to show that $h(t,w)$ relaxes toward a local equilibrium state $h^\eq(t,w)$ of \eqref{eq100}, i.e. a distribution function canceling out the Fokker--Planck flux $Q_p(h)+Q_\ell(h)$, which has the form
        \begin{equation} \label{eq:local equilibrium}
            h^\eq(t,w) = C(t) (1-w)^{-1+\frac{\lambda_p(\omega_p+\varepsilon)^2}{\sigma_p^2 + \sigma_\ell^2} (1-m_w(t))+\frac{\lambda_\ell(\omega_\ell+\varepsilon)}{\sigma_p^2 + \sigma_\ell^2} (1-\mu_\ell)} (1+w)^{-1+\frac{\lambda_p(\omega_p+\varepsilon)^2}{\sigma_p^2 + \sigma_\ell^2} (1+m_w(t))+\frac{\lambda_\ell(\omega_\ell+\varepsilon)}{\sigma_p^2 + \sigma_\ell^2} (1+\mu_\ell)},
        \end{equation}
    with normalization constant explicitly given by $C(t) = 1/B(b_1(t),b_2(t))$, where $B$ is again the Beta function
    \begin{equation*}
        B(b_1(t),b_2(t)) = \int_I (1-w)^{-1+b_1(t)} (1+w)^{-1+b_2(t)} \dd w,
    \end{equation*}
    of time-dependent parameters $b_1(t)$ and $b_2(t)$ given by
    \begin{equation*}
    \begin{split}
        b_1(t) &= \frac{\lambda_p(\omega_p+\varepsilon)^2}{\sigma_p^2 + \sigma_\ell^2} (1-m_w(t))+\frac{\lambda_\ell(\omega_\ell+\varepsilon)}{\sigma_p^2 + \sigma_\ell^2} (1-\mu_\ell), \\[2mm]
        b_2(t) &= \frac{\lambda_p(\omega_p+\varepsilon)^2}{\sigma_p^2 + \sigma_\ell^2} (1+m_w(t))+\frac{\lambda_\ell(\omega_\ell+\varepsilon)}{\sigma_p^2 + \sigma_\ell^2} (1+\mu_\ell).
    \end{split}
    \end{equation*}
    The convergence will then follow easily, since Proposition \ref{prop6} ensures that $m_w(t) \underset{t \to +\infty}{\longrightarrow} \mu_\ell$ and thus the local equilibrium $h^\eq(t,w)$ itself relaxes toward $h^\infty(w)$ when $t \to +\infty$. In particular, the regularity of $h^\eq(t,w)$ can be inferred similarly to that of $h^\infty(w)$. As long as both exponents of the local equilibrium are nonnegative, then it belongs to any $L^q(I)$, $q \in [1,+\infty]$. Conversely, whenever $\displaystyle \inf_{t \in \R_+}\{-1+b_1(t), -1+b_2(t)\} < 0$, we deduce that $h^\eq \in L^\infty(\R_+,L^q(I))$ for any $1 \leq q < \bar{q}$, where
    \begin{equation} \label{eq:Lq local equilibrium}
        \bar{q} = \left( 1 - \frac{\lambda_p(\omega_p+\varepsilon)^2}{\sigma_p^2 + \sigma_\ell^2} \left(1-\sup_{t \in \R_+}|m_w(t)|\right) - \frac{\lambda_\ell(\omega_\ell+\varepsilon)}{\sigma_p^2 + \sigma_\ell^2} (1-|\mu_\ell|)  \right)^{-1}. 
    \end{equation}
    Notice that $\displaystyle \sup_{t \in \R_+} |m_w(t)| \leq  \max\{|m_w^\init|, |\mu_\ell|\}$, since $m_w(t)$ converges exponentially to $\mu_\ell$ thanks again to Proposition \ref{prop6}.
    
    To proceed, we thus consider the relative entropy between the distributions $h(t,w)$ and $h^\eq(t,w)$
    \begin{equation*}
        \mathcal{H}(h|h^\eq)(t) = \int_I h(t,w) \log \frac{h(t,w)}{h^\eq(t,w)} \dd w \geq 0,
    \end{equation*}
    and we evaluate its time derivative, recovering initially
    \begin{equation*}
        \begin{split}
            \frac{\dd}{\dd t} \mathcal{H}(h|h^\eq)(t) &= \int_I \left( 1 + \log \frac{h(t,w)}{h^\eq(t,w)} \right) \partial_t h(t,w) \dd w - \int_I h(t,w) \partial_t \log h^\eq(t,w) \dd w \\[2mm]
            &= \mathcal{T}_1(t) + \mathcal{T}_2(t).
        \end{split}
    \end{equation*}
    For the first integral, the computations from \cite{20} give
    \begin{equation} \label{eq:estimate on T1}
        \mathcal{T}_1(t) = \int_I \left( 1 + \log \frac{h(t,w)}{h^\eq(t,w)} \right) \Big( Q_p(h)(t,w) + Q_\ell(h)(t,w) \Big) \dd w = -\mathcal{I}_H(h|h^\eq)(t),
    \end{equation}
    where the last term denotes the entropy production 
    \begin{equation*}
        \mathcal{I}_H(h|h^\eq)(t) = 4 \int_I (1 - w^2) h^\eq(t,w) \left( \partial_w \sqrt{\frac{h(t,w)}{h^\eq(t,w)}} \right)^2 \dd w \geq 0
    \end{equation*}
    of $\mathcal{H}(h|h^\eq)(t)$. Next, we develop the derivative inside the integrand of $\mathcal{T}_2(t)$ obtaining
    \begin{equation*}
        \partial_t \log h^\eq(t,w) = \frac{C'(t)}{C(t)} + \frac{\lambda_p(\omega_p+\varepsilon)^2}{\sigma_p^2 + \sigma_\ell^2} \frac{\dd}{\dd t} m_w(t) \log \frac{1+w}{1-w},
    \end{equation*}
    with
    \begin{equation*}
        \frac{C'(t)}{C(t)} = -\frac{\lambda_p(\omega_p+\varepsilon)^2}{\sigma_p^2 + \sigma_\ell^2} C(t) \frac{\dd}{\dd t} m_w(t) \int_I \log \frac{1+w}{1-w} (1-w)^{-1+b_1(t)} (1+w)^{-1+b_2(t)} \dd w.
    \end{equation*}
    From \eqref{eq:evolution of mean}, we infer that
    \begin{equation*}
        \frac{\dd}{\dd t} m_w(t) = \frac{\dd}{\dd t} (m_w(t) - \mu_\ell) \leq \eta e^{-\tau t},
    \end{equation*}
    where $\eta, \tau > 0$ are explicitly given by $\eta = \lambda_\ell(\omega_\ell+\varepsilon)\left|m_w^\init - \mu_\ell\right|$ and $\tau=\lambda_\ell (\omega_\ell + \eps)$. Therefore, by suitably updating the value of $\eta$ to include the other constants, we deduce that $C'(t) / C(t) \leq \eta e^{-\tau t}$ and then
    \begin{equation*}
        \mathcal{T}_2(t) \leq \eta \left( 1 + \int_I |\log(1-w)| h(t,w) \dd w + \int_I |\log(1+w)| h(t,w) \dd w \right) e^{-\tau t},
    \end{equation*}
     given that $\int_I h(t,w) \dd w = 1$. By assumption, there exists $q > 1$ for which $h \in L^\infty(\R_+, L^q(I))$. Using that the functions $|\log(1-w)|$ and $|\log(1+w)|$ belong to $L^{q'}(I)$ for any $q' \in [1,+\infty)$, we apply H\"older's inequality with conjugate exponents $q$ and $q'=\frac{q}{q-1}$, and infer the upper bounds
    \begin{equation*}
        \begin{split}
            \int_I \log (1-w) h(t,w) \dd w &\leq \norm{\log(1-w)}_{L^{q'}(I)} \norm{h}_{L^\infty(\R_+, L^q(I))}, \\[2mm]
            \int_I \log (1+w) h(t,w) \dd w &\leq \norm{\log(1+w)}_{L^{q'}(I)} \norm{h}_{L^\infty(\R_+, L^q(I))},
        \end{split}
    \end{equation*}
    which allow to finally control $\mathcal{T}_2(t)$ as
    \begin{equation} \label{eq:estimate on T2}
        \mathcal{T}_2(t) \leq \eta \left( 1 + \norm{\log(1+w)}_{L^{q'}(I)} \norm{h}_{L^\infty(\R_+, L^q(I))} \right) e^{-\tau t}.
    \end{equation}
    Note that obviously $\norm{\log(1-w)}_{L^{q'}(I)} = \norm{\log(1+w)}_{L^{q'}(I)}$ and we have once again renamed $\eta$ to account for all successive upper bounds of the constants. Combining estimates \eqref{eq:estimate on T1} and \eqref{eq:estimate on T2}, we finally deduce that
    \begin{equation} \label{eq:evolution of H}
        \frac{\dd}{\dd t} \mathcal{H}(h|h^\eq)(t) \leq -\mathcal{I}_H(h|h^\eq)(t) + \eta \left( 1 + \norm{\log(1+w)}_{L^{q'}(I)} \norm{h}_{L^\infty(\R_+, L^q(I))} \right) e^{-\tau t}.
    \end{equation}

    We proceed by introducing the Hellinger distance
    \begin{equation*}
        \mathcal{D}(h|h^\eq)(t)  = \left( \int_I \left( \sqrt{h(t,w)} - \sqrt{h^\eq(t,w)} \right)^2 \dd w \right)^{\frac{1}{2}},
    \end{equation*}
    between $h(t,w)$ and $h^\eq(t,w)$, for which holds the fundamental relation
    \begin{equation} \label{eq:relation between D and H}
        \mathcal{D}^2(h|h^\eq)(t) \leq \frac{1}{2} \mathcal{I}_H(h|h^\eq)(t),
    \end{equation}
    proven in \cite{FurPulTerTos}. Moreover, its temporal evolution reads
    \begin{equation*}
        \begin{split}
            \frac{\dd}{\dd t} \mathcal{D}^2(h|h^\eq)(t) &= \int_I \left( 1 - \sqrt{\frac{h^\eq(t,w)}{h(t,w)}} \right) \partial_t h(t,w) \dd w + \int_I \left( 1 - \sqrt{\frac{h(t,w)}{h^\eq(t,w)}} \right) \partial_t h^\eq(t,w) \dd w \\[2mm]
            &= \mathcal{T}_3(t) + \mathcal{T}_4(t).
        \end{split}
    \end{equation*}
    Now, the first term is linked to the entropy production
    \begin{equation*}
        \mathcal{I}_D(h|h^\eq) = 8 \int_I (1 - w^2) h^\eq(t,w) \left( \partial_w \sqrt[4]{\frac{h(t,w)}{h^\eq(t,w)}} \right)^2 \dd w \geq 0,
    \end{equation*}
    of the Hellinger distance, through the identity \cite{FurPulTerTos}
    \begin{equation} \label{eq:estimate on T3}
        \mathcal{T}_3(t) = \int_I \left( 1 - \sqrt{\frac{h^\eq(t,w)}{h(t,w)}} \right) \Big( Q_p(h)(t,w) + Q_\ell(h)(t,w) \Big) \dd w = -\mathcal{I}_D(h|h^\eq)(t).
    \end{equation}
    Moreover, $\mathcal{T}_4(t)$ can be recast as
    \begin{equation*}
        \begin{split}       
        \mathcal{T}_4(t) & = \int_I \left(h^\eq(t,w) - \sqrt{h(t,w) h^\eq(t,w)}\right) \partial_t \log h^\eq(t,w) \dd w \\ 
        & = \int_I h^\eq(t,w)   \partial_t \log h^\eq(t,w) \dd w  - \int_I \sqrt{h(t,w) h^\eq(t,w)}    \partial_t \log h^\eq(t,w) \dd w,
        \end{split}
    \end{equation*}
    from which we initially deduce, using the previous computations on $\partial_t \log h^\eq(t,w)$ to deal with the first integral and the Cauchy--Schwarz inequality to deal with the second one, that
    \begin{equation*}
        \begin{split}
            \mathcal{T}_4(t) &\leq \eta e^{-\tau t} \left( 1 + \int_I |\log(1-w)| h^\eq(t,w) \dd w + \int_I |\log(1+w)| h^\eq(t,w) \dd w \right) \\[2mm]
            & \hspace*{2cm} + \underbrace{\left( \int_I h(t,w) \dd w \right)^{\frac{1}{2}}}_{=1} \left( \int_I h^\eq(t,w) \left(\partial_t \log h^\eq(t,w) \right)^2 \dd w \right)^{\frac{1}{2}}.
        \end{split}
    \end{equation*}
    Recalling then the minimal regularity of the local equilibrium, we fix $1 < q < \bar{q}$ from \eqref{eq:Lq local equilibrium} such that $h^\eq \in L^\infty(\R_+,L^q(I))$, and we apply H\"older's inequality with exponents $q$ and $q' = \frac{q}{q-1}$ to control the functions $|\log(1-w)|$, $\log^2(1-w)$, $|\log(1+w)|$, and $\log^2(1+w)$, which obviously all belong to any $L^{q'}(I)$, $q' \geq 1$. Simple computations allow to recover the estimate
    \begin{equation} \label{eq:estimate on T4}
        \mathcal{T}_4(t) \leq \eta \left( 1 + \norm{\log(1+w)}_{L^{q'}(I)} \norm{h^\eq}_{L^\infty(\R_+, L^q(I))} + \norm{\log^2(1+w)}_{L^{q'}(I)}^{1/2} \norm{h^\eq}_{L^\infty(\R_+, L^q(I))}^{1/2} \right) e^{-\tau t},
    \end{equation}
    where the value of $\eta$ has been suitably redefined to include all dependencies on the upper bounding constants. The combination of inequalities \eqref{eq:estimate on T3} and \eqref{eq:estimate on T4} finally yields the estimate
    \begin{equation} \label{eq:evolution of D}
        \frac{\dd}{\dd t} \mathcal{D}^2(h|h^\eq)(t) + \frac{\dd}{\dd t} \left( \frac{\tilde{\eta}}{\tau} e^{-\tau t} \right) \leq -\mathcal{I}_D(h|h^\eq)(t),
    \end{equation}
    where we have conveniently denoted
    \begin{equation*}
        \tilde{\eta} = \eta \left( 1 + \norm{\log(1+w)}_{L^{q'}(I)} \norm{h^\eq}_{L^\infty(\R_+, L^q(I))} + \norm{\log^2(1+w)}_{L^{q'}(I)}^{1/2} \norm{h^\eq}_{L^\infty(\R_+, L^q(I))}^{1/2} \right).
    \end{equation*}

    Now, integrating \eqref{eq:evolution of H} over $t \in \R_+$ and using that $\mathcal{H}(h|h^\eq)(t)$ is nonnegative for any $t \in \R_+$, we infer the $L^1(\R_+)$ integrability of $\mathcal{I}_H(h|h^\eq)(t)$ since
    \begin{equation} \label{eq:integrability of I}
        \int_{\R_+} \mathcal{I}_H(h|h^\eq)(s)   \dd s \leq \zeta \big(1 + \mathcal{H}(h|h^\eq)(0) \big)
    \end{equation}
    for some $\zeta > 0$, from which also follows $\mathcal{D}^2(h|h^\eq) \in L^1(\R_+)$ thanks to relation \eqref{eq:relation between D and H}. Moreover, inequality \eqref{eq:evolution of D} tells us that the quantity $\mathcal{D}^2(h|h^\eq)(t) + \frac{\tilde{\eta}}{\tau} e^{-\tau t}$ is decreasing in time and, since it belongs to $L^1(\R_+)$, the latter must decay at infinity at least like $\smallO\left(t^{-1}\right)$. These considerations finally lead to the decay estimate
    \begin{equation*}
        \norm{h-h^\eq}_{L^1(I)} \leq 2 \mathcal{D}(h|h^\eq)(t) = \smallO\left(t^{-1/2}\right), \quad t \to +\infty,
    \end{equation*}
    where the inequality linking the $L^1$ norm with the Hellinger distance is obtained using the simple relation $|a - b| = \left|\sqrt{a} - \sqrt{b}\right|\left(\sqrt{a} + \sqrt{b}\right)$ and two successive applications of Cauchy--Schwarz inequality. We deduce the relaxation of $h(t,w)$ toward the local equilibrium $h^\eq(t,w)$.

    To conclude, it is enough to prove that $h^\eq(t,w)$ itself approaches the global equilibrium $h^\infty(w)$ when $t \to +\infty$. This follows for example from an investigation of their relative entropy $\mathcal{H}(h^\infty|h^\eq)(t)$. We begin by estimating the latter as
    \begin{equation} \label{eq:relative entropy between equilibria}
        \begin{split}
            \mathcal{H}&(h^\infty|h^\eq)(t) = \int_I h^\infty(w) \log \frac{h^\infty(w)}{h^\eq(t,w)} \dd w \\[2mm]
            &= \log \frac{C^\infty}{C(t)} \underbrace{\int_I h^\infty(w) \dd w}_{=1} + (b_1^\infty-b_1(t)) \int_I \log(1-w) h^\infty(w) \dd w + (b_2^\infty-b_2(t)) \int_I \log(1+w) h^\infty(w) \dd w \\[2mm]
            & \leq \frac{C^\infty-C(t)}{C(t)} + \eta \norm{\log(1+w)}_{L^{q'}(I)} \norm{h^\infty}_{L^q(I)} e^{-\tau t},
        \end{split}
    \end{equation}
    applying the bounds $\log x \leq x-1$ for any $x > 0$, $b_i^\infty-b_i(t) \leq \frac{\lambda_p(\omega_p+\varepsilon)^2}{\sigma_p^2 + \sigma_\ell^2} |m_w(t) - \mu_\ell|$ for $i=1,2$, and H\"older's inequality as before to bound the last two integrals. It thus only remains to control the difference $C^\infty-C(t)$. Using the algebraic relation $ab-cd = \frac{1}{2}(a-c)(b+d) + \frac{1}{2}(a+c)(b-d)$, we find
    \begin{equation} \label{eq:difference of constants}
        \begin{split}
            C^\infty-C(t) &= C^\infty C(t) \Big(B(b_1(t),b_2(t)) - B(b_1^\infty,b_2^\infty) \Big) \\[4mm]
            &= \frac{C^\infty C(t)}{2} \int_I \left( (1-w)^{-1+b_1(t)} - (1-w)^{-1+b_1^\infty} \right)\\
            & \qquad\qquad\qquad\qquad \times \left( (1+w)^{-1+b_2(t)} + (1+w)^{-1+b_2^\infty} \right) \dd w \\[4mm]
            &\qquad + \frac{C^\infty C(t)}{2} \int_I \left( (1-w)^{-1+b_1(t)} + (1-w)^{-1+b_1^\infty} \right) \left( (1+w)^{-1+b_2(t)} - (1+w)^{-1+b_2^\infty} \right) \dd w \\[4mm]
            &= \frac{C^\infty C(t)}{2} \int_I (1-w)^{-1+b_1^\infty} \left( (1-w)^{b_1(t)-b_1^\infty} - 1 \right) \left( (1+w)^{-1+b_2(t)} + (1+w)^{-1+b_2^\infty} \right) \dd w \\[4mm]
            &\qquad + \frac{C^\infty C(t)}{2} \int_I \left( (1-w)^{-1+b_1(t)} + (1-w)^{-1+b_1^\infty} \right) (1+w)^{-1+b_2^\infty} \left( (1+w)^{b_2(t)-b_2^\infty} - 1 \right) \dd w,
        \end{split}
    \end{equation}
    and we expect the differences $(1-w)^{b_1(t)-b_1^\infty} - 1$ and $(1+w)^{b_2(t)-b_2^\infty} - 1$ to vanish exponentially in the limit $t \to +\infty$. We prove this fact for the first difference (the second will follow similarly), by rewriting it as
    \begin{equation*}
        \begin{split}
            (1-w)^{b_1(t)-b_1^\infty} - 1 &= (b_1(t)-b_1^\infty) \log(1-w) \int_0^1 (1-w)^{s(b_1(t)-b_1^\infty)} \dd s \\[2mm]
            &= \xi e^{-\tau t} \log(1-w) \int_0^1 (1-w)^{s \xi e^{-\tau t}} \dd s,
        \end{split}
    \end{equation*}
    where $\xi = -\frac{\lambda_p(\omega_p+\varepsilon)^2}{\sigma_p^2 + \sigma_\ell^2} \left(m_w^\init-\mu_\ell\right)$ and $\tau > 0$ are two constants, and the sign of $\xi$ is determined by that of the initial condition $m_w^\init-\mu_\ell$. The delicate case occurs when $\xi < 0$, since one needs to control an integral term of the form $\int_I \int_0^1 (1-w)^{-s |\xi| e^{-\tau t}} \dd s   \dd w$, which could be unbounded depending on the parameters of the problem (notice that the integration in $s$ does not cause any additional problem). However, the presence of the exponential decay $e^{-\tau t}$ ensures that this issue is prevented for $t$ large enough. In particular, for any given $q' > 1$ one can find a threshold time $T > 0$ such that the function $g(w) = \log(1-w) (1-w)^{s \xi e^{-\tau t}}$ belongs to $L^{q'}(I)$ for any $s \in [0,1]$ and $t \geq T$. Let us recall that there exists $q > 1$ such that $h^\eq \in L^\infty(\R_+,L^q(I))$ and $h^\infty \in L^q(I)$. Then obviously all factors appearing in the integrands of $C(t)$ and $C^\infty$ share the same $L^q(I)$ regularity. Thus, we can apply H\"older's inequality with conjugates $q$ and $q' = \frac{q}{q-1}$ to each of the integral terms of \eqref{eq:difference of constants}, in this way
    \begin{equation*}
        \begin{split}
            \int_I (1-w)^{-1+b_1^\infty} & \left( (1-w)^{b_1(t)-b_1^\infty} - 1 \right) (1+w)^{-1+b_2(t)} \\[2mm]
            &\leq |\xi| e^{-\tau t} \int_I \left| (1-w)^{-1+b_1^\infty} (1+w)^{-1+b_2(t)} \right| \left| \log(1-w) \int_0^1 (1-w)^{s \xi e^{-\tau t}} \dd s \right| \dd w \\[2mm]
            &\leq |\xi| \norm{(1-w)^{-1+b_1^\infty} (1+w)^{-1+b_2(t)}}_{L^q(I)} \norm{\log(1-w) \int_0^1 (1-w)^{s \xi e^{-\tau t}} \dd s}_{L^{q'}(I)} e^{-\tau t},
        \end{split}
    \end{equation*}
    for any $s \in [0,1]$ and $t \geq T$ chosen to guarantee the $L^{q'}(I)$ integrability of the quantity $\log(1-w) \int_0^1 (1-w)^{s \xi e^{-\tau t}} \dd s$. The remaining terms are treated in the same way.

    Therefore, going back to inequality \eqref{eq:relative entropy between equilibria}, we can infer the bound
    \begin{equation*}
        \mathcal{H}(h^\infty|h^\eq)(t) \leq \eta e^{-\tau t}, \quad t \geq T,
    \end{equation*}
    for a suitable constant $\eta > 0$ that can be determined combining the previous estimates. Thanks to the Csiszár--Kullback--Pinsker inequality, we thus deduce that $h^\eq(t,w)$ converges toward $h^\infty$ as
    \begin{equation*}
        \norm{h^\eq-h^\infty}_{L^1(I)}^2 \leq 2 \mathcal{H}(h^\infty|h^\eq)(t) \leq \eta e^{-\tau t}, \quad t \to +\infty,
    \end{equation*}
    which concludes the proof of our result, since
    \begin{equation*}
        \norm{h-h^\infty}_{L^1(I)} \leq \norm{h-h^\eq}_{L^1(I)} + \norm{h^\eq-h^\infty}_{L^1(I)} = \smallO\left(t^{-1/2}\right) + \smallO\left(e^{-\tau t}\right) = \smallO\left(t^{-1/2}\right),
    \end{equation*}
    when $t \to +\infty$.
    \end{proof}

    Theorem \ref{teo:conv} requires that $h \in L^\infty(\R_+, L^q(I))$ for some $q > 1$, in order to exploit H\"older's inequality. As previously observed, the limit steady state $h^\infty(w)$ (which is known \textit{a priori} as soon as the main parameters of the problem are fixed) belongs to $L^q(I)$ for some $q > 1$, hence one expects that the solution $h(t,w)$ to equation \eqref{eq100} shares the same regularity of $h^\infty(w)$, as long as one assumes it for the initial datum. Unfortunately, this is not obvious to demonstrate. The following lemma ensures however that a weaker version of this fact is true, in the particular case of consensus formation, where $h^\infty \in L^\infty(I)$. As we will show, this property will be enough to prove the convergence of $h(t,w)$ toward $h^\infty(w)$, provided that we impose some additional reasonable constraints on the parameters of the problem.   

    \begin{lemma} \label{lemma:norm}
        Let $q \in [2, \infty)$ and consider a solution $h(t,w)$ of Eq. \eqref{eq100}, completed by the no-flux boundary conditions \eqref{eq:main BC}. Assume that following additional no-flux boundary conditions hold for any $t>0$:
        \begin{equation} \label{eq:additional BC}
            \begin{cases}
                h(t,w)\Big|_{w=\pm 1} = 0,  \\[4mm] 
                (1-w^2) h^{q-1}(t,w) \partial_w h(t,w) \Big|_{w=\pm 1} = 0. 
            \end{cases}
        \end{equation}
        If initially $h^\init \in L^q(I)$, then $h(t,\cdot) \in L^q(I)$ for any $t>0$. More precisely,
        \begin{equation} \label{norm_estimate}
            \norm{h(t, \cdot)}_{L^q(I)} \leq \eta_q e^{\tau_q t} \quad \forall t > 0,
        \end{equation}
        where $\eta_q = \norm{h^{\init}}_{L^q(I)}$ and $\tau_q = \frac{(q-1)\left( \sigma_p^2 + \sigma_\ell^2 \right)}{q} \left(\frac{\lambda_p (\omega_p + \eps)^2 + \lambda_\ell (\omega_\ell + \eps)}{\sigma_p^2 + \sigma_\ell^2} - 1\right)$.
    \end{lemma}

    \begin{proof}
        Adapting the computations from \cite[Proposition 2]{2}, we successively write 
        \begin{equation*}
        \begin{split}
            \int_I & h^{q-1}(t,w) \partial_t h(t,w)   \dd w \\[2mm]
            &= \int_I h^{q-1}(t,w) \Bigg( \frac{\sigma_p^2 + \sigma_\ell^2}{2} \partial^2_w \Big((1-w^2) h(t,w)\Big) \\
            & \hspace*{3cm} + \partial_w \Big( \lambda_p(\omega_p + \eps)^2 (w-m_w(t)) h(t,w)\Big) + \partial_w \Big(\lambda_\ell (\omega_\ell + \eps) (w-\mu_\ell) h(t,w) \Big) \Bigg) \dd w \\[2mm] 
            &= \frac{\sigma_p^2 + \sigma_\ell^2}{2} \int_I h^{q-1}(t,w) \partial_w \Big((1-w^2) \partial_w h(t,w)\Big)   \dd w \\[2mm] 
            & \qquad + \int_I h^{q-1}(t,w) \partial_w \Big(\Big( \big(\lambda_p (\omega_p+\eps)^2 + \lambda_\ell (\omega_\ell + \eps) - \sigma_p^2 - \sigma_\ell^2\big) w - (\lambda_p m_w(t) + \lambda_\ell \mu_\ell)\Big) h(t,w)\Big) \dd w \\[4mm] 
            &= \mathcal{T}_1 + \mathcal{T}_2,
        \end{split}
        \end{equation*}
        by performing one derivation in the diffusion term. Integrating by parts $\mathcal{T}_1$ and using the second of the boundary conditions \eqref{eq:additional BC}, we get 
        \begin{equation*}
            \mathcal{T}_1 = - \frac{\sigma_p^2 + \sigma_\ell^2}{2} (q-1) \int_I (1-w^2) h^{q-2}(t,w) \big( \partial_w h(t,w) \big)^2   \dd w \le 0, 
        \end{equation*}
        since $h(t,w)$ is nonnegative. Concerning $\mathcal{T}_2$, we manipulate it in two different ways. We can perform an integration by parts using the first of the boundary conditions \eqref{eq:additional BC}, to get 
        \begin{equation*}
            \mathcal{T}_2 = -(q-1) \int_I \Big(\left(\lambda_p (\omega_p+\eps)^2 + \lambda_\ell (\omega_\ell + \eps) - \sigma_p^2 - \sigma_\ell^2\right)w - (\lambda_p m_w(t) + \lambda_\ell \mu_\ell)\Big) h^{q-1}(t,w) \partial_w h(t,w)   \dd w.
        \end{equation*}
        Alternatively, we can compute the derivative with respect to $w$, to obtain 
        \begin{equation*}
        \begin{split}
            \mathcal{T}_2 &= \big(\lambda_p (\omega_p + \eps)^2 + \lambda_\ell (\omega_\ell + \eps) - \sigma_p^2 - \sigma_\ell^2\big) \int_I h^q (t,w)   \dd w \\[2mm] 
            & + \int_I \Big( \big(\lambda_p (\omega_p+\eps)^2 + \lambda_\ell (\omega_\ell + \eps) - \sigma_p^2 - \sigma_\ell^2\big)w - (\lambda_p m_w(t) + \lambda_\ell \mu_\ell)\Big) h^{q-1}(t,w) \partial_w h(t,w)   \dd w. 
        \end{split}
        \end{equation*}
        By splitting $\mathcal{T}_2 = \frac{1}{q} \mathcal{T}_2 + \frac{q-1}{q} \mathcal{T}_2$ and applying the first computation to $\frac{1}{q} \mathcal{T}_2$ and the second to $\frac{q-1}{q} \mathcal{T}_2$, we thus end up with the estimate 
        \begin{equation*}
            \int_I h^{q-1}(t,w) \partial_t h(t,w)   \dd w \leq \frac{q-1}{q} \left(\lambda_p (\omega_p + \eps)^2 + \lambda_\ell (\omega_\ell + \eps) - \sigma_p^2 - \sigma_\ell^2\right) \norm{h(t, \cdot)}^q_{L^q(I)},  
        \end{equation*}
        implying that 
        \begin{equation*}
            \frac{\dd}{\dd t}\norm{h(t, \cdot)}_{L^q(I)}^q \leq (q-1)\left( \sigma_p^2 + \sigma_\ell^2 \right) \left(\frac{\lambda_p (\omega_p + \eps)^2 + \lambda_\ell (\omega_\ell + \eps)}{\sigma_p^2 + \sigma_\ell^2} - 1\right) \norm{h(t, \cdot)}_{L^q(I)}^q,
        \end{equation*}
        from which \eqref{norm_estimate} follows using Gr\"onwall's lemma. 
    \end{proof}

    \begin{remark}
        The only global equilibria \eqref{h_inf} compatible with the boundary conditions \eqref{eq:additional BC} are the Beta distributions $h^\infty(w)$ having both exponents nonnegative (consensus formation inside the population). This implies in particular that the constants $\tau_q$ are in fact nonnegative. Moreover, they are also equi-bounded by $\lambda_p (\omega_p + \eps)^2 + \lambda_\ell (\omega_\ell + \eps) - \sigma_p^2 - \sigma_\ell^2$, for any $q \in [2,+\infty)$. Therefore, in the limit $q \to +\infty$ we infer the regularity $h(t,\cdot) \in L^\infty(I)$ for any $t > 0$, specifically
        \begin{equation*}
            \norm{h(t, \cdot)}_{L^\infty(I)} \leq \eta_\infty e^{\tau_\infty t} \quad \forall t > 0,
        \end{equation*}
        with $\eta_\infty = \norm{h^\init}_{L^\infty(I)}$ and $\tau_\infty = \left( \sigma_p^2 + \sigma_\ell^2 \right) \left(\frac{\lambda_p (\omega_p + \eps)^2 + \lambda_\ell (\omega_\ell + \eps)}{\sigma_p^2 + \sigma_\ell^2} - 1\right)$, provided that $h^\init \in L^\infty(I)$.
    \end{remark}
    
    We highlight that Lemma \ref{lemma:norm} does not imply the uniform-in-time regularity $h \in L^\infty (\R_+, L^q(I))$ for $q \geq 2$, assumed in Theorem \ref{teo:conv}. Nonetheless, it is still possible to prove that $h(t,w)$ converges toward $h^\infty(w)$ as long as 
    \begin{equation} \label{eq:tau}
        0 \leq \frac{q-1}{q}\left(\lambda_p (\omega_p + \eps)^2 + \lambda_\ell (\omega_\ell + \eps) - \sigma_p^2 - \sigma_\ell^2\right) = \tau_q < \tau = \lambda_\ell (\omega_\ell + \eps), 
    \end{equation}
    by adapting the proof of Theorem \ref{teo:conv}, at the instances where the quantity $\norm{h}_{L^\infty(\R_+,L^q(I))}$ is used to estimate the time derivative of the relative entropy $\mathcal{H}(h|h^\eq)(t)$. This occurs when deriving the bound \eqref{eq:estimate on T2}, and the assumption $\norm{h}_{L^\infty(\R_+,L^q(I))} < +\infty$ is there to ensure the $L^1(\R_+)$ integrability of the term $\mathcal{T}_2(t)$ from which follows inequality \eqref{eq:integrability of I}. However, the same integrability could be recovered as soon as $\norm{h(t,\cdot)}_{L^q(I)}$ satisfies the weaker control \eqref{norm_estimate}, with $\tau_q$ from \eqref{eq:tau}. In particular, inequality \eqref{eq:estimate on T2} would need to be changed into
    \begin{equation*}
        \mathcal{T}_2(t) \leq \eta \left( 1 + \norm{\log(1+w)}_{L^{q'}(I)} \norm{h^\init}_{L^q(I)} \right) e^{-(\tau-\tau_q) t},
    \end{equation*}
    and the rest of the proof would follow easily.
    
    \begin{remark}
        We observe from condition \eqref{eq:tau} that by choosing $q=2$, we can impose the least restrictive constraints on the main parameters of our model, in order to recover the convergence of $h(t,w)$ toward the global equilibrium $h^\infty(w)$ through Theorem \ref{teo:conv}. Indeed, for this particular choice condition \eqref{eq:tau} reads
        \begin{equation*}
             - \lambda_\ell (\omega_\ell + \eps) + \sigma_\ell^2 \leq \lambda_p (\omega_p + \eps)^2 - \sigma_p^2 < \lambda_\ell (\omega_\ell + \eps) + \sigma_\ell^2, 
        \end{equation*}
        and is satisfied, for example, as soon as the leaders' influence is stronger than the individuals' one in driving the consensus dynamics.
    \end{remark} 

\subsection{Additional analytical properties of the full model} \label{evol3}

    We conclude the analysis of our models with a focus on the properties of the whole distribution $f(t,A,w)$. Even if considering the controlled model and assuming $G \equiv 1$ and $D(w)=\sqrt{1-w^2}$, its marginal distributions $g(t,A)$ and $h(t,w)$ converge, the convergence of $f(t,A,w)$ is not guaranteed.  However, if such convergence result holds and $f(t,A,w) \underset{t \to +\infty}{\longrightarrow} f^\infty(A,w)$, then the marginal distributions of $f^\infty(A,w)$ must be the distributions $g^\infty(A)$ and $h^\infty(w)$ previously introduced. In particular, in the absence of opinion leaders, we expect $f(t,A,w)$ to converge in the sense of distributions toward the global equilibrium of the Fokker--Planck equation \eqref{eq53}, given by 
    \begin{equation*} \label{inf}
        f^\infty (A,w) = C  \delta(A-A_c^*)   (1-w)^{-1+\frac{(\omega_p+\varepsilon)^2}{\nu_p}(1-m_w^\init)}(1+w)^{-1+\frac{(\omega_p+\varepsilon)^2}{\nu_p}(1+m_w^\init)},   
    \end{equation*}
    where $C > 0$ is the same normalization constant appearing in (\ref{h_inf}). Again, observe that the influence of the initial condition $f^\init(A,w)$ is only seen through the preserved initial average opinion $m_w^\init$, and not in the equilibrium distribution of the activity level, which is independent of the initial datum. An analogous equilibrium can be readily derived in the presence of opinion leaders, where the influence of the initial distribution is instead lost. Finally, we stress that for the uncontrolled model \eqref{eq29}, $f(t,A,w)$ cannot reach an equilibrium state, since its marginal $g(t,A)$ does not converge to an equilibrium (assuming that initially not all agents have an activity level equal to $A_p^*$). 

    Nonetheless, we can still recover the following fundamental analytical properties of nonnegativity and uniqueness for the distribution $f(t,A,w)$, without requiring that $G \equiv 1$ and $D(w)=\sqrt{1-w^2}$. We point out that whenever we are considering the controlled model in what follows, we will tacitly assume that the initial distribution has compact support (see Subsection \ref{sec:mean}). 

    \begin{proposition}\label{nonneg}
        Let $f(t,A,w)$ be a solution of the uncontrolled model (\ref{eq29}) or of the controlled one (\ref{eq53}) and consider an initial datum $f^\init \in L^1(\R\times I)$ such that $f^\init(A,w) \geq 0$ for a.e. $A \in \R$ and $w \in I$. Then, $f(t,A,w) \geq 0$ for a.e. $A \in \R$, $w \in I$, and for any $t > 0$. 
    \end{proposition}
    \begin{proof}
    We will only prove the result for model \eqref{eq29}, focusing on the operator $Q_p(f,f)$, as similar computations can be performed to study both $Q_\ell(f,f_\ell)$ and model \eqref{eq53}. More precisely, we consider the evolution equation
    \begin{equation} \label{eq:positivity initial}
        \begin{split}
            \partial_t f(t,A,w) &= \frac{\sigma_p^2}{2} \partial^2_w\left(D^2(w)f(t,A,w)\right) +\lambda_p(\bar{A}\omega_p+\varepsilon)\partial_w\left(\mathcal{K}[f](t,w)f(t,A,w)\right) \\[2mm] 
            & \qquad -\lambda_A \partial_A\Big((\bar{A}\omega_p+\varepsilon-a_p) f(t,A,w)\Big),
        \end{split} 
    \end{equation}
    completed with the appropriate no-flux boundary conditions \eqref{eq35}, and we recall that the operator $\mathcal{K}[f](t,w)$ reads
    \begin{equation*}
        \mathcal{K}[f](t,w) = \int_{\R\times I} (\bar{B}\omega_p+\varepsilon)G(w,v)(w-v) f(t,B,v)   \dd B   \dd v.
    \end{equation*}
    Given a parameter $\epsilon > 0$, we consider a regularized increasing approximation of the sign function $\textrm{sign}_\epsilon(u)$, $u \in \R$, and define the corresponding regularization $|f|_\epsilon(t,A,w)$ of $|f|(t,A,w)$ to be the primitive of $\textrm{sign}_\epsilon(f)(t,A,w)$ for any $(A,w) \in \R\times I$. We then multiply both sides of the above equation by $\textrm{sign}_\epsilon(f)$, integrate with respect to $w$ and $A$ and apply integration by parts to successively get
    \begin{equation*}
        \begin{split}
            \frac{\dd}{\dd t} \int_{\R\times I} |f|_\epsilon(t,A,w)   &\dd A   \dd w = - \lambda_p \int_{\R\times I}(\bar{A}\omega_p+\varepsilon)\mathcal{K}[f] \textrm{sign}_\epsilon'(f)(t,A,w) f(t,A,w) \partial_w f(t,A,w) \dd A \dd w \\[2mm]
            & -\frac{\sigma_p^2}{2} \int_{\R\times I} \partial_w D^2(w) \textrm{sign}_\epsilon'(f)(t,A,w) f(t,A,w) \partial_w f(t,A,w) \dd A \dd w \\[2mm]
            & -\frac{\sigma_p^2}{2} \int_{\R\times I} \textrm{sign}_\epsilon'(f)(t,A,w) \Big(\partial_w f(t,A,w)\Big)^2 D^2(w) \dd A \dd w \\[2mm]
            & +\lambda_A \int_{\R\times I} (\bar{A}\omega_p+\varepsilon-a_p) \textrm{sign}_\epsilon'(f)(t,A,w) f(t,A,w) \partial_A f(t,A,w) \dd A \dd w,
        \end{split} 
    \end{equation*}
    where the two middle lines were obtained by developing the derivative $\partial_w \Big(D^2(w)f(t,A,w)\Big)$. Now, noticing that we can rewrite
    \begin{equation*}
        \begin{split}
            & \textrm{sign}_\epsilon'(f) f \partial_w f = \partial_w\Big(\textrm{sign}_\epsilon(f)f - |f|_\epsilon\Big), \\[2mm]
            & \textrm{sign}_\epsilon'(f) f \partial_A f = \partial_A\Big(\textrm{sign}_\epsilon(f)f - |f|_\epsilon\Big),
        \end{split}
    \end{equation*}
    we can integrate back by parts the three terms containing these factors (with respect to $A$ the first one, with respect to $w$ the other two) and take the limit $\epsilon \to 0$, to prove that they all vanish. Since the second to last term is nonpositive for any $\eps$, we thus conclude that the $L^1(\R\times I)$ norm of $f(t,A,w)$ is non-increasing, namely
    \begin{equation*}
        \frac{\dd}{\dd t} \int_{\R\times I} |f|(t,A,w) \dd A \dd w \leq 0.
    \end{equation*}
    Introducing now the negative part $f_\epsilon^-(t,A,w)$ of the solution via the same regularization as
    \begin{equation*}
        f_\epsilon^-(t,A,w)=\frac{1}{2} \Big(|f|_\epsilon (t,A,w) - f(t,A,w)\Big), 
    \end{equation*}
    we exploit the fact that all operators on the right-hand side of equation \eqref{eq:positivity initial} are mass-preserving to infer that
    \begin{equation*}
        \begin{split}
            \frac{\dd}{\dd t} \int_{\R\times I} f_\epsilon^-(t,A,w) \dd A \dd w &= \frac{1}{2} \frac{\dd}{\dd t} \int_{\R\times I} |f|_\epsilon(t,A,w) \dd A \dd w - \frac{1}{2} \frac{\dd}{\dd t} \int_{\R\times I} f(t,A,w) \dd A \dd w \\[2mm]
            &= \frac{1}{2} \frac{\dd}{\dd t} \int_{\R\times I} |f|_\epsilon(t,A,w) \dd A \dd w,
        \end{split}
    \end{equation*}
    and that in the limit $\epsilon \to 0$ it holds
    \begin{equation*}
        \frac{\dd}{\dd t} \int_{\R\times I} f^-(t,A,w) \dd A \dd w \leq 0,
    \end{equation*}
    for all $t > 0$, concluding that the solution $f(t,A,w)$ remains indeed nonnegative. 
    \end{proof}

    \begin{remark}
        From Proposition \ref{nonneg} it follows that the $L^1(\R\times I)$ regularity of the initial datum is preserved, since the model is mass-preserving. Note that we cannot expect more than $L^1$ regularity, because the transport operator along the activity leads to the formation of a Dirac delta in the limit $ t \to +\infty$.
    \end{remark}   

    \begin{proposition}\label{uniq}
        Let us consider two solutions $f_1(t,A,w)$ and $f_2(t,A,w)$ of either the uncontrolled model \eqref{eq29} or of the controlled one \eqref{eq53}, and assume that $\mathcal{K}[f_1](t,w) = \mathcal{K}[f_2](t,w)$ for a.e. $w \in  I$ and any $t > 0$. For any initial conditions $f_1^\init, f_2^\init \in L^1(\R\times I)$, if $f_1^\init(A,w) = f_2^\init(A,w)$ for a.e. $A \in \R$ and $w \in  I$, then $f_1(t,A,w) = f_2(t,A,w)$ for a.e. $A \in \R$, $w \in  I$, and for any $t > 0$. 
    \end{proposition}
    
    \begin{proof}
        Let us still prove the result for model \eqref{eq29} and for the operator $Q_p(f,f)$ only, which amounts to consider once again equation \eqref{eq:positivity initial}. The assumption that the integral operators $\mathcal{K}[f_1]$ and $\mathcal{K}[f_2]$ coincide almost everywhere obviously leads to a linearization of the model. Therefore, one deduces that the difference of two solutions $f = f_1 - f_2$ solves
        \begin{equation} \label{eq:uniqueness initial}
        \begin{split}
            \partial_t f(t,A,w) &= \frac{\sigma_p^2}{2} \partial^2_w\Big(D^2(w) f(t,A,w)\Big) +\lambda_p(\bar{A}\omega_p+\varepsilon)\partial_w\Big(\mathcal{K}[f_1](t,w) f(t,A,w)\Big) \\[2mm] 
            & \qquad -\lambda_A \partial_A\Big((\bar{A}\omega_p+\varepsilon-a_p) f(t,A,w)\Big),
        \end{split} 
    \end{equation}
    and satisfies initially $f^\init(A,w) = 0$ for a.e. $A \in \R$ and $w \in  I$. We can then proceed as in the proof of Proposition \ref{nonneg}, by considering again an increasing regularization of the sign function, depending on a parameter $\epsilon > 0$, and using it to define an approximation $|f|_\epsilon(t,A,w)$ of $|f|(t,A,w)$. After multiplication by $\textrm{sign}_\epsilon(f)(t,A,w)$ on both sides of \eqref{eq:uniqueness initial}, and applying the same reasoning as before, one recovers the estimate
    \begin{equation*}
        \frac{\dd}{\dd t} \int_{\R\times I} |f|(t,A,w) \dd A \dd w \leq 0,
    \end{equation*}
    in the limit $\epsilon \to 0$. Since the $L^1(\R\times I)$ norm of $f(t,A,w)$ does not increase over time and $\norm{f^\init}_{L^1(\R\times I)} = 0$, we conclude that $\norm{f}_{L^1(\R\times I)} = 0$ for any $t > 0$ and thus $f_1(t,A,w) = f_2(t,A,w)$ for a.e. $A \in \R$ and $w \in  I$, and for any $t > 0$.
    \end{proof}

    We conclude by providing a sufficient condition which ensures that $\mathcal{K}[f_1](t,w) = \mathcal{K}[f_2](t,w)$ is satisfied for a.e. $w \in  I$ and any $t > 0$. 

    \begin{lemma} \label{lemma1}
        Consider two solutions $f_1(t,A,w)$ and $f_2(t,A,w)$ of either the uncontrolled model \eqref{eq29} or of the controlled one \eqref{eq53}, such that their corresponding initial conditions satisfy $f_1^\init(A,w) = f_2^\init(A,w)$ for a.e. $A \in \R$ and $w \in  I$. If the function $\mathcal{K}[f](t,w)$ is analytical in $t$ for a.e. $w \in I$, then $\mathcal{K}[f_1](t,w) = \mathcal{K}[f_2](t,w)$ for a.e. $w \in  I$ and any $t > 0$. 
    \end{lemma}

    \begin{proof}
        Since the initial distributions coincide, it suffices to show that 
        \begin{equation*}
            \partial^n_t \Big(\mathcal{K}[f_1](t,w) - \mathcal{K}[f_2](t,w)\Big)  \Big|_{t=0} = 0,
        \end{equation*}
        for any $n>0$ and for a.e. $w \in I$. For the sake of clarity, we will only prove the result for model \eqref{eq29} and focusing on the operator $Q_p(f,f)$, as similar computations can be performed to study both $Q_\ell(f,f_\ell)$ and model \eqref{eq53}. The no-flux boundary conditions \eqref{eq35} imply 
        \begin{equation*}
        \begin{split}
            \partial_t & (\mathcal{K}[f_1](t,w) - \mathcal{K}[f_2](t,w))  \Big|_{t=0} = \int_{\R \times I} (\bar{B}\omega_p + \eps) G(w,v) (w-v) \partial_t\Big(f_1(t,B,v)-f_2(t,B,v)\Big) \dd B   \dd v \Bigg|_{t=0} \\[2mm] 
            =& \int_{\R \times I} (\bar{B}\omega_p + \eps) \partial_v\Big(G(w,v)(w-v)\Big) \left( \frac{\sigma_p^2}{2} \partial_v\Big(D^2(v)(f_2^\init(B,v)-f_1^\init(B,v))\Big)\right) \dd B   \dd v \\[2mm] 
            & + \int_{\R \times I} (\bar{B}\omega_p + \eps) \partial_v\Big(G(w,v)(w-v)\Big) \lambda_p (\bar{B}\omega_p + \eps) \Big(\mathcal{K}[f_2^\init](v) f_2^\init(B,v)-\mathcal{K}[f_1^\init](v) f_1^\init(B,v)\Big) \dd B   \dd v \\[2mm]
            & + \int_{\R \times I} \partial_B (\bar{B}\omega_p + \eps) G(w,v) (w-v) \lambda_A (\bar{B}\omega_p + \eps - a_p) \Big(f_1^\init(B,v)-f_2^\init(B,v)\Big) \dd B   \dd v \\[4mm]
            =&\ 0,
        \end{split}
        \end{equation*}
        for a.e. $w \in I$, since the initial distributions coincide. The computations for higher values of $n > 1$ are similar. 
    \end{proof}
    
\section{Conclusions} \label{concl}

    \noindent In this work we proposed a kinetic approach to model opinion dynamics among a population of agents that have a certain propensity to interact with each other, based on their variable level of activity. The population was divided into three different classes: active agents representing the socially active part of the population, who were likely to interact; inactive agents having no propensity to interact and discuss with others (for example, in a political setting these agents could be thought of as part of the nowadays widespread phenomenon of abstentionism \cite{12, 14, 13}), who had a small probability to interact; undecided agents being in the middle between active and inactive individuals. 
    
    The model highlighted that social interactions are of fundamental importance for preventing opinion polarization and the emergence of extreme opinions. The model showed that, in order to reach a consensus, compromise dynamics must be more influential than self-thinking ones. Moreover, we proved that an effective compromise process among active agents could prevent the formation of extreme opinions, while the absence of compromise among inactive individuals could instead lead to polarization. In particular, it is not possible to observe consensus formation among inactive agents and, at the same time, opinion polarization among active ones. 

    We presented a simple, but effective control strategy to increase the number of active individuals and decrease the number of inactive ones. The idea of the strategy was that each agent could suddenly start to show interest in a certain topic (for example, politics), therefore rapidly increasing the number of their social interactions in this regard. Similarly, an individual could lose interest in that topic, thus reducing the number of their related interactions. We showed that if the latter two processes were correctly balanced, the strategy was indeed effective, but if such balance was not met, an opposite effect might occur. We remark that this strategy reproduces a completely realistic phenomenon that could manifest in any society, hence it does not limit its possible range of applications in a social sciences framework. 

    We also studied the effects of additional interactions with opinion leaders \cite{2}, that could be regarded as the media or highly influential individuals. Through these external factors, we were able to control the evolution of the average opinion inside the population, ensuring that it converges toward the average opinion of the leaders. In particular, influencing the opinion of active agents was easier than influencing the opinion of inactive ones, since the latter were more reluctant to interact. This could be interpreted as the fact that it is easier to inform the segment of the population that remains informed and watches, for example, informative television programs. Note that it could also be possible to consider leaders that spread conspiracy theories and contribute to disinformation. However, in such case it would be reasonable to assume that not all the population would be equally influenced by them (i.e., $G \not\equiv 1$), meaning that the average opinion of the population would not converge toward the average opinion of the leaders (as one would expect). 
    
    Finally, we carried out a mathematical analysis to investigate the large time behavior of the marginal distributions of the solution $f(t,A,w)$ to the controlled model. At first, we studied the evolution of the activity level independently of the opinion dynamics. We were thus able to prove a result on the weak convergence to equilibrium for the marginal $g(t,A)$, toward a Dirac delta. In particular, we showed that all agents tend to become active in finite time, simplifying the structure of the Fokker--Planck equation that described the evolution of the other marginal $h(t,w)$, and allowing us to prove a result on the strong convergence to equilibrium for $h(t,w)$, toward a Beta distribution. Assuming that socially active individuals would tend to form a consensus, this demonstrated that through our control strategy it was possible to suppress opinion polarization. The most interesting scenario occurred when both interactions agent--agent and agent--leaders were considered. Indeed, in this situation the average opinion of the population was not preserved over time and the system evolved under a Fokker--Planck equation whose adjoint was given by a Wright--Fisher-type model with time-dependent coefficients \cite{CheStr,EpsMaz1,EpsMaz2}. The proof of our this result involved the use of the relative entropy method \cite{20} to recover the rate of convergence to equilibrium. 

    Future research will concern numerical simulations of our models, suitably extended to include an underlying social structure (described for example through a graphon \cite{3}), that could allow us to consider the formation of clusters within the population and to identify the most socially active groups \cite{1}.
    
\bigskip
\bigskip
\noindent \textbf{Acknowledgments.} The authors are members and acknowledge the support of {\it Gruppo Nazionale di Fisica Matematica} (GNFM) of {\it Istituto Nazionale di Alta Matematica} (INdAM). AB acknowledges the support of the European Union’s Horizon Europe research and innovation programme, under the Marie Skłodowska-Curie grant agreement No. 101110920, project MesoCroMo (A Mesoscopic approach to Cross-diffusion Modelling in population dynamics). AB also acknowledges the support of the INdAM--GNFM project CUP E5324001950001 (Multi-species non-Maxwellian Fokker--Planck models inferred from local non-equilibrium distributions). JB acknowledges the support of the University of Parma through the action \textit{Bando di Ateneo 2022 per la ricerca}, co-funded by MUR-Italian Ministry of University and Research - D.M. 737/2021 - PNR - PNRR - NextGenerationEU (project ``Collective and Self-Organised Dynamics: Kinetic and Network Approaches''). JB also thanks the support of the project PRIN 2022 PNRR ``Mathematical Modelling for a Sustainable Circular Economy in Ecosystems'' (project code P2022PSMT7, CUP D53D23018960001) funded by the European Union - NextGenerationEU, PNRR-M4C2-I 1.1, and by MUR-Italian Ministry of University and Research. 

\bigskip
\bigskip
\noindent \textbf{Disclaimer.} Funded by the European Union. Views and opinions expressed are however those of the author(s) only and do not necessarily reflect those of the European Union or of the European Research Executive Agency (REA). Neither the European Union nor the granting authority can be held responsible for them.

\begin{figure}[h!]
\begin{flushleft}
\includegraphics[scale=0.3]{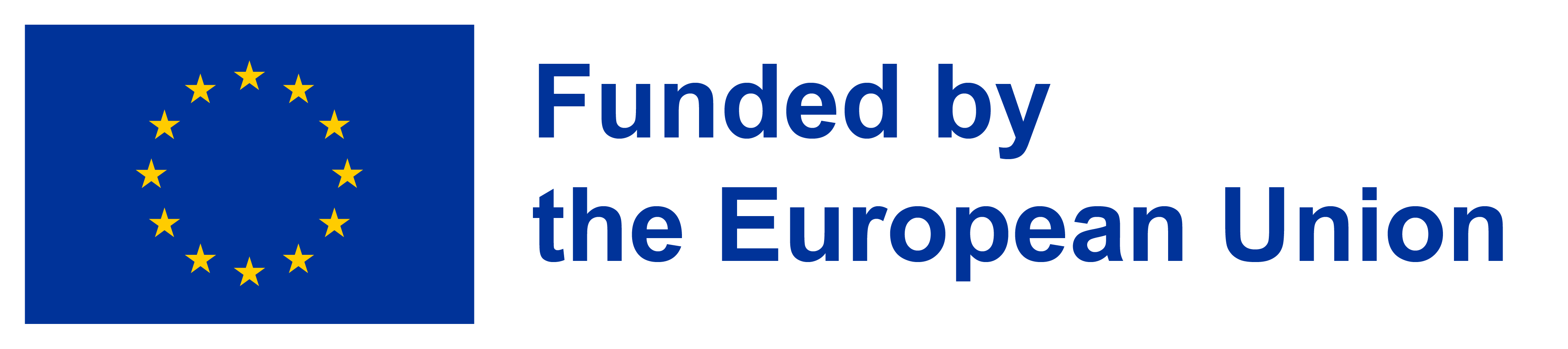}
\end{flushleft}
\end{figure}

\bigskip
\noindent \textbf{Funding.} This research has been supported by the European Union’s Horizon Europe research and innovation programme (project ``MesoCroMo – A Mesoscopic approach to Cross-diffusion Modelling in population dynamics'', Marie Skłodowska-Curie grant agreement No. 101110920), by INdAM--GNFM (project ``Multi-species non-Maxwellian Fokker--Planck models inferred from local non-equilibrium distributions'', CUP E5324001950001), by the European Union - NextGenerationEU and by the MUR-Italian Ministry of University and Research (project “Collective and Self-Organised Dynamics: Kinetic and Network Approaches”, action “Bando di Ateneo 2022 per la ricerca” D.M. 737/2021 - PNR - PNRR, and project “Mathematical Modelling for a Sustainable Circular Economy in Ecosystems”, PRIN 2022 PNRR, code P2022PSMT7, CUP D53D23018960001).


\bigskip
\nocite{*}
\bibliographystyle{plain}
\bibliography{Bibliography_OA}

@article {16,
    AUTHOR = {Zanella, M.},
     TITLE = {Kinetic models for epidemic dynamics in the presence of
              opinion polarization},
   JOURNAL = {Bull. Math. Biol.},
    VOLUME = {85},
      YEAR = {2023},
    NUMBER = {5},
     PAGES = {1--29},
       note = {\url{https://doi.org/10.1007/s11538-023-01147-2}},
}

@article{2,
    author = {Bondesan, A. and Toscani, G. and Zanella, M.},
    title = {Kinetic compartmental models driven by opinion dynamics: Vaccine hesitancy and social influence},
    journal = {Math. Mod. Meth. Appl. S.},
    volume = {34},
    number = {6},
    pages = {1043--1076},
    year = {2024},
    note = {\url{https://doi.org/10.1142/S0218202524400062}}}

@book {10,
    AUTHOR = {Cercignani, C.},
     TITLE = {The {B}oltzmann equation and its applications},
    SERIES = {Applied Mathematical Sciences},
    VOLUME = {67},
 PUBLISHER = {Springer-Verlag, New York},
      YEAR = {1988},
       note = {\url{https://doi.org/10.1007/978-1-4612-1039-9}}
}

@article {11,
    AUTHOR = {Cordier, S. and Pareschi, L. and Toscani, G.},
     TITLE = {On a kinetic model for a simple market economy},
   JOURNAL = {J. Stat. Phys.},
    VOLUME = {120},
      YEAR = {2005},
     PAGES = {253--277},
       note = {\url{https://doi.org/10.1007/s10955-005-5456-0}}
}

@article {12,
    AUTHOR = {Feddersen, T. J. and Pesendorfer, W.},
     TITLE = {Abstention in Elections with Asymmetric Information and Diverse Preferences},
   JOURNAL = {Am. Polit. Sci. Rev.},
    VOLUME = {93},
    number={2}, 
    pages={381--398}, 
    note={\url{https://doi.org/10.2307/2585402}}, 
      YEAR = {1999}}

@article {1,
    AUTHOR = {Moussa{\"\i}d, M.},
     TITLE = {Opinion formation and the collective dynamics of risk perception},
   JOURNAL = {{PLOS ONE}},
    VOLUME = {8},
    NUMBER={12},  
    pages={e84592}, 
      YEAR = {2013}, 
    note={\url{https://doi.org/10.1371/journal.pone.0084592}}}

@article {3,
    AUTHOR = {D\"uring, B. and Franceschi, J. and Wolfram,
              M.-T. and Zanella, M.},
     TITLE = {Breaking consensus in kinetic opinion formation models on
              graphons},
   JOURNAL = {J. Nonlinear Sci.},
    VOLUME = {34},
      YEAR = {2024},
    NUMBER = {79},
       note = {\url{https://doi.org/10.1007/s00332-024-10060-4}}
}

@article {5,
    AUTHOR = {Glasscock, D.},
     TITLE = {What is \dots{} a graphon?},
   JOURNAL = {Notices Amer. Math. Soc.},
    VOLUME = {62},
      YEAR = {2015},
    NUMBER = {1},
     PAGES = {46--48},
    note={\url{http://dx.doi.org/10.1090/noti1204}}
}

@article {13,
    AUTHOR = {Legnante, G. and Segatti, P.},
     TITLE = {Intermittent abstentionism and multi-level mobilisation in {I}taly},
   JOURNAL = {Modern Italy},
    VOLUME = {14},
    NUMBER = {2},
    PAGES = {167--181},
    note={\url{https://doi.org/10.1080/13532940902871802}}, 
      YEAR = {2009}}

@book {9,
    AUTHOR = {Pareschi, L. and Toscani, G.},
     TITLE = {Interacting Multiagent Systems: Kinetic Equations and Monte Carlo Methods},
 PUBLISHER = {Oxford University Press},
      YEAR = {2013},
       note = {\url{https://global.oup.com/academic/product/interacting-multiagent-systems-9780199655465?cc=it&lang=en&}}
}

@article {15,
    AUTHOR = {Pareschi, L. and Toscani, G.},
     TITLE = {Wealth distribution and collective knowledge: a {B}oltzmann approach},
   JOURNAL = {Philos. Trans. R. Soc. Lond. Ser. A Math. Phys. Eng. Sci.},
    VOLUME = {372},
      YEAR = {2014},
     PAGES = {20130396},
       note = {\url{https://doi.org/10.1098/rsta.2013.0396}}
}

@article {6,
    AUTHOR = {Pareschi, L. and Toscani, G. and Tosin, A. and
              Zanella, M.},
     TITLE = {Hydrodynamic models of preference formation in multi-agent societies},
   JOURNAL = {J. Nonlinear Sci.},
    VOLUME = {29},
      YEAR = {2019},
    NUMBER = {6},
     PAGES = {2761--2796},
       note = {\url{https://doi.org/10.1007/s00332-019-09558-z}}
}

@article {8,
    AUTHOR = {Toscani, G.},
     TITLE = {Kinetic models of opinion formation},
   JOURNAL = {Commun. Math. Sci.},
    VOLUME = {4},
      YEAR = {2006},
    NUMBER = {3},
     PAGES = {481--496},
       note = {\url{https://doi.org/10.4310/cms.2006.v4.n3.a1}}
}

@article {14,
    AUTHOR = {Feltrin, P. and Ieraci, G.},
     TITLE = {Facilitating voting and electoral participation in {I}taly. {O}n some possible measures to contrast involuntary and imposed abstentionism},
   JOURNAL = {Rivista di Digital Politics},
    VOLUME = {1},
PAGES = {29--48},
note = {\url{https://www.rivisteweb.it/doi/10.53227/107475}}, 
      YEAR = {2023}}

@article {20,
    AUTHOR = {Auricchio, F. and Toscani, G. and Zanella, M.},
     TITLE = {Trends to equilibrium for a nonlocal {F}okker--{P}lanck equation},
   JOURNAL = {Appl. Math. Lett.},
    VOLUME = {145},
      YEAR = {2023},
      PAGES = {108746}, 
    note={\url{https://doi.org/10.1016/j.aml.2023.108746}}}

@article{22,
  title={Sociophysics: {A} new approach of sociological collective behaviour: I. {M}ean-behaviour description of a strike},
  author={Galam, S. and Gefen (Feigenblat), Y. and Shapir, Y.},
  journal={Journal of Mathematical Sociology},
  year={1982},
  volume={9}, 
  number={1}, 
  pages={1--13}, 
    note={\url{https://doi.org/10.1080/0022250X.1982.9989929}}
}

@article{23,
author = {Ochrombel, R.},
year = {2001},
pages = {1091},
title = {Simulation of {S}znajd Sociophysics Model with Convincing Single Opinions},
volume = {12},
number = {7},
pages = {1091--1091},
journal = {Int. J. Mod. Physics. C}, 
note={\url{https://doi.org/10.1142/S0129183101002346}}
}

@article{24,
author = {Stauffer, D.},
year = {2002},
title = {Percolation and Galam Theory of Minority Opinion Spreading},
volume = {13},
number = {7},
pages = {975--977},
journal = {Int. J. Mod. Physics. C}, 
note={\url{https://doi.org/10.1142/S0129183102003735}}
}

@article{25,
author = {Stauffer, D. and Oliveira, P.},
year = {2002},
pages = {587--592},
title = {Persistence of opinion in the {S}znajd consensus model: {C}omputer simulation},
volume = {30},
journal = {Eur. Phys. J. B}, 
pages = {587--592},
note={\url{https://doi.org/10.1140/epjb/e2002-00418-0}}
}

@article{26,
author = {Sznajd-Weron, K. and Sznajd, J.},
title = {Opinion evolution in closed community},
journal = {Int. J. Mod. Physics. C},
volume = {11},
number = {6},
pages = {1157--1165},
year = {2000}, 
note={\url{https://doi.org/10.1142/S0129183100000936}}
}

@article{27,
title = {Sociodynamics––a systematic approach to mathematical modelling in the social sciences},
journal = {Chaos Solitons Fract.},
volume = {18},
number = {3},
pages = {431--437},
year = {2003},
author = {Weidlich, W.}, 
note={\url{https://doi.org/10.1016/S0960-0779(02)00666-5}}
}

@article{28,
author = {Ulsh\"ofer, C. and Renkewitz, F. and Betsch, T.},
year = {2011},
pages = {742--753},
title = {The Influence of Narrative v. Statistical Information on Perceiving Vaccination Risks},
volume = {31},
number = {5},
journal = {Med. Decis. Making}, 
note={\url{https://doi.org/10.1177/0272989X11400419}}
}

@article {g1,
    AUTHOR = {Borgs, C. and Chayes, J.T. and Cohn, H. and
              Zhao, Y.},
     TITLE = {An {$L^p$} theory of sparse graph convergence {II}: {LD}
              convergence, quotients and right convergence},
   JOURNAL = {Ann. Probab.},
    VOLUME = {46},
      YEAR = {2018},
    NUMBER = {1},
     PAGES = {337--396},
       note = {\url{https://doi.org/10.1214/17-AOP1187}}
}

@article {g2,
    AUTHOR = {Borgs, C. and Chayes, J.T. and Cohn, H. and
              Zhao, Y.},
     TITLE = {An {$L^p$} theory of sparse graph convergence {I}: {L}imits,
              sparse random graph models, and power law distributions},
   JOURNAL = {Trans. Amer. Math. Soc.},
    VOLUME = {372},
      YEAR = {2019},
    NUMBER = {5},
     PAGES = {3019--3062},
       note = {\url{https://doi.org/10.1090/tran/7543}}
}

@article {FurPulTerTos,
    AUTHOR = {Furioli, G. and Pulvirenti, A. and Terraneo, E. and Toscani, G.},
     TITLE = {Fokker-{P}lanck equations in the modelling of socio-economic phenomena},
   JOURNAL = {Math. Models Methods Appl. Sci.},
    VOLUME = {27},
      YEAR = {2017},
      NUMBER = {1},
     PAGES = {115--158}, 
    note={\url{https://doi.org/10.1142/S0218202517400048}}}

@article {MR4020534,
    AUTHOR = {Furioli, G. and Pulvirenti, A. and Terraneo, E. and Toscani, G.},
     TITLE = {Wright--{F}isher-type equations for opinion formation, large time behavior and weighted logarithmic-{S}obolev inequalities},
   JOURNAL = {Ann. Inst. H. Poincar\'e{} C Anal. Non Lin\'eaire},
    VOLUME = {36},
      YEAR = {2019},
    NUMBER = {7},
     PAGES = {2065--2082},
       note = {\url{https://doi.org/10.1016/j.anihpc.2019.07.005}},
}

@article {EpsMaz1,
    AUTHOR = {Epstein, C.L. and Mazzeo, R.},
     TITLE = {Wright--{F}isher diffusion in one dimension},
   JOURNAL = {SIAM J. Math. Anal.},
    VOLUME = {42},
      YEAR = {2010},
    NUMBER = {2},
     PAGES = {568--608},
       note = {\url{https://doi.org/10.1137/090766152}},
}

@book {EpsMaz2,
    AUTHOR = {Epstein, C.L. and Mazzeo, R.},
     TITLE = {Degenerate diffusion operators arising in population biology},
    SERIES = {Annals of Mathematics Studies},
    VOLUME = {185},
 PUBLISHER = {Princeton University Press, Princeton, NJ},
      YEAR = {2013},
       note = {\url{https://doi.org/10.1515/9781400846108}},
}

@article {CheStr,
    AUTHOR = {Chen, L. and Stroock, D.W.},
     TITLE = {The fundamental solution to the {W}right--{F}isher equation},
   JOURNAL = {SIAM J. Math. Anal.},
    VOLUME = {42},
      YEAR = {2010},
    NUMBER = {2},
     PAGES = {539--567},
       note = {\url{https://doi.org/10.1137/090764207}},
}

@article {MR3945372,
    AUTHOR = {Albi, G. and Bongini, M. and Rossi, F. and
              Solombrino, F.},
     TITLE = {Leader formation with mean-field birth and death models},
   JOURNAL = {Math. Models Methods Appl. Sci.},
    VOLUME = {29},
      YEAR = {2019},
    NUMBER = {4},
     PAGES = {633--679},
       note = {\url{https://doi.org/10.1142/S0218202519400025}},
}

@article {MR3945374,
    AUTHOR = {Gualandi, S. and Toscani, G.},
     TITLE = {Human behavior and lognormal distribution. {A} kinetic description},
   JOURNAL = {Math. Models Methods Appl. Sci.},
    VOLUME = {29},
      YEAR = {2019},
    NUMBER = {4},
     PAGES = {717--753},
       note = {\url{https://doi.org/10.1142/S0218202519400049}},
}

@article {MR4093946,
    AUTHOR = {Furioli, G. and Pulvirenti, A. and Terraneo, E. and Toscani, G.},
     TITLE = {Non-{M}axwellian kinetic equations modeling the dynamics of wealth distribution},
   JOURNAL = {Math. Models Methods Appl. Sci.},
    VOLUME = {30},
      YEAR = {2020},
    NUMBER = {4},
     PAGES = {685--725},
       note = {\url{https://doi.org/10.1142/S0218202520400023}},
}

@article {DurWol,
    AUTHOR = {D\"uring, B. and Wolfram, M.-T.},
     TITLE = {Opinion dynamics: inhomogeneous {B}oltzmann-type equations
              modelling opinion leadership and political segregation},
   JOURNAL = {Proc. A.},
  FJOURNAL = {Proceedings A},
    VOLUME = {471},
      YEAR = {2015},
    NUMBER = {2182},
     PAGES = {20150345, 21},
      ISSN = {1364-5021,1471-2946},
   MRCLASS = {91D10 (35Q20 35Q91)},
  MRNUMBER = {3420842},
       DOI = {10.1098/rspa.2015.0345},
       note = {\url{https://doi.org/10.1098/rspa.2015.0345}},
}

@article {DegLiuMerTar,
    AUTHOR = {Degond, P. and Liu, J.-G. and Merino-Aceituno, S. and
              Tardiveau, T.},
     TITLE = {Continuum dynamics of the intention field under weakly
              cohesive social interaction},
   JOURNAL = {Math. Models Methods Appl. Sci.},
  FJOURNAL = {Mathematical Models and Methods in Applied Sciences},
    VOLUME = {27},
      YEAR = {2017},
    NUMBER = {1},
     PAGES = {159--182},
      ISSN = {0218-2025,1793-6314},
   MRCLASS = {82C21 (35Q84 82C22 82C26 82C31 82C40 91D10)},
  MRNUMBER = {3597011},
       DOI = {10.1142/S021820251740005X},
       note = {\url{https://doi.org/10.1142/S021820251740005X}},
}

@article {DimPerTosZan,
    AUTHOR = {Dimarco, G. and Perthame, B. and Toscani, G. and Zanella, M.},
     TITLE = {Kinetic models for epidemic dynamics with social
              heterogeneity},
   JOURNAL = {J. Math. Biol.},
  FJOURNAL = {Journal of Mathematical Biology},
    VOLUME = {83},
      YEAR = {2021},
    NUMBER = {1},
     PAGES = {Paper No. 4, 32},
      ISSN = {0303-6812,1432-1416},
   MRCLASS = {92D30 (35Q84 35Q92 92D25)},
  MRNUMBER = {4278446},
       DOI = {10.1007/s00285-021-01630-1},
       note = {\url{https://doi.org/10.1007/s00285-021-01630-1}},
}

@article {BouSal,
    AUTHOR = {Boudin, L. and Salvarani, F.},
     TITLE = {A kinetic approach to the study of opinion formation},
   JOURNAL = {M2AN Math. Model. Numer. Anal.},
  FJOURNAL = {M2AN. Mathematical Modelling and Numerical Analysis},
    VOLUME = {43},
      YEAR = {2009},
    NUMBER = {3},
     PAGES = {507--522},
      ISSN = {0764-583X,1290-3841},
   MRCLASS = {91D30 (82C22)},
  MRNUMBER = {2536247},
       DOI = {10.1051/m2an/2009004},
       note = {\url{https://doi.org/10.1051/m2an/2009004}},
}

@article {DelLoyTos,
    AUTHOR = {Della Marca, R. and Loy, N. and Tosin, A.},
     TITLE = {An {SIR}-like kinetic model tracking individuals' viral load},
   JOURNAL = {Netw. Heterog. Media},
  FJOURNAL = {Networks and Heterogeneous Media},
    VOLUME = {17},
      YEAR = {2022},
    NUMBER = {3},
     PAGES = {467--494},
      ISSN = {1556-1801,1556-181X},
   MRCLASS = {92D30 (60J74 82C40)},
  MRNUMBER = {4421535},
       DOI = {10.3934/nhm.2022017},
       note = {\url{https://doi.org/10.3934/nhm.2022017}},
}

@article {BisSpiTos,
    AUTHOR = {Bisi, M. and Spiga, G. and Toscani, G.},
     TITLE = {Kinetic models of conservative economies with wealth
              redistribution},
   JOURNAL = {Commun. Math. Sci.},
  FJOURNAL = {Communications in Mathematical Sciences},
    VOLUME = {7},
      YEAR = {2009},
    NUMBER = {4},
     PAGES = {901--916},
      ISSN = {1539-6746,1945-0796},
   MRCLASS = {91B60 (35B40 35Q20 82C40)},
  MRNUMBER = {2604625},
       DOI = {10.4310/cms.2009.v7.n4.a5},
       note = {\url{https://doi.org/10.4310/cms.2009.v7.n4.a5}},
}
\bigskip
\bigskip

\setlength\parindent{0pt}

\end{document}